\documentclass[11pt,a4paper]{amsart}

\usepackage[margin=1.25in,a4paper]{geometry}
\usepackage{lmodern}
\usepackage[utf8]{inputenc}
\usepackage[T1]{fontenc}

\usepackage{microtype}
\usepackage[shortlabels]{enumitem}

\usepackage[english]{babel}
\usepackage{mathrsfs}
\usepackage{amssymb}

\usepackage{xcolor}
\definecolor{linkcol}{RGB}{0,80,158}
\definecolor{citecol}{RGB}{46,117,120}

\usepackage[colorlinks]{hyperref}
\hypersetup{
    colorlinks=true,
    linkcolor=linkcol,
    urlcolor=linkcol,
    citecolor=citecol,
}
\usepackage[hyperpageref]{backref}
\usepackage[msc-links,alphabetic]{amsrefs}
\usepackage[capitalize,nameinlink]{cleveref}
\crefname{equation}{}{}
\numberwithin{equation}{section}

\usepackage{tikz} 
\usepackage{pgfplots} 
\pgfplotsset{compat=1.16}
\usepackage{tikz-cd}
\tikzcdset{arrow style=tikz, arrows={thick}, diagrams={>=stealth}}

\theoremstyle{plain}
\newtheorem{theorem}{Theorem}[section]
\newtheorem{lemma}[theorem]{Lemma}
\newtheorem{corollary}[theorem]{Corollary}

\theoremstyle{definition}
\newtheorem*{definition}{Definition}

\newtheorem{question}{Question}
\theoremstyle{remark}
\newtheorem*{remark}{Remark}

\usepackage{thm-restate}

\newcommand{\bC}{\mathbb{C}}
\newcommand{\bD}{\mathbb{D}}
\newcommand{\bQ}{\mathbb{Q}}
\newcommand{\bR}{\mathbb{R}}
\newcommand{\bT}{\mathbb{T}}
\newcommand{\bZ}{\mathbb{Z}}
\newcommand{\bS}{\mathbb{S}}
\newcommand{\cE}{\mathcal{E}}

\newcommand{\sC}{\mathscr{C}}

\newcommand{\sH}{\mathscr{H}}
\newcommand{\sJ}{\mathscr{J}}
\newcommand{\sM}{\mathscr{M}}
\newcommand{\sN}{\mathscr{N}}

\newcommand{\sZ}{\mathscr{Z}}
\newcommand{\fp}{\mathfrak{p}}
\newcommand{\abs}[1]{\lvert#1\rvert}
\newcommand{\Abs}[1]{\left\lvert#1\right\rvert}
\newcommand{\babs}[1]{\bigl\lvert#1\bigr\rvert}
\newcommand{\norm}[1]{\lVert#1\rVert}

\renewcommand{\Re}{\operatorname{Re}}
\renewcommand{\Im}{\operatorname{Im}}
\newcommand{\dd}{\mathrm{d}}
\newcommand{\dist}{\operatorname{dist}}
\newcommand{\diam}{\operatorname{diam}}

\title[Littlewood--Paley and mean counting for Dirichlet series]{The Littlewood--Paley formula and mean counting function for vertical limits of Dirichlet series}
\author{Viktor Andersson}
\address{Department of Mathematical Sciences, Norwegian University of Science and Technology (NTNU), 7491 Trondheim, Norway}
\email{viktor.andersson@ntnu.no}
\date{\today}

\begin{document}
\begin{abstract}
We prove a Littlewood--Paley formula for the Hardy space of Dirichlet series $\sH^p$ with $1\leq p<\infty$ in terms of almost every vertical limit function. This significantly strengthens previous results, which hold either only as an average over the vertical limit functions or under additional assumptions of uniform convergence. As part of our approach, we obtain a Hardy--Stein identity for the derivative of the $p$-mean of almost every vertical limit. We further show that the mean counting function exists for any $f$ in $\sH^p$ in terms of almost all of its vertical limit functions. This is done by establishing a version of Jensen's formula in this setting. In the process, we also deduce ergodic versions of Fatou's lemma and the monotone and dominated convergence theorems for the Kronecker flow.
\end{abstract}
\thanks{The author is supported by Grant 354537 of the Research Council of Norway.}
\maketitle
\section{Introduction}\label{section:introduction}
Before going into the details, let us outline our contributions. The main results of this paper are canonical versions of the Littlewood--Paley formula for the Hardy space of Dirichlet series $\sH^p$ with $1\leq p<\infty$ and the mean counting function for any $f\in\sH^p$. Up to this point, there have been two types of results in this direction: those that hold as an average over the vertical limit functions $f_\chi$, and those that hold under additional assumptions of uniform convergence. Our results improve on both of these by establishing versions that hold for $f_\chi$ for almost every $\chi$ without any additional assumptions on $f$. In this sense, our results are essentially optimal.

Let us now provide the details. For $1\leq p<\infty$, let $\sH^p$ denote the Hardy space of Dirichlet series, defined as the completion of the space of all Dirichlet polynomials $f(s)=\sum_{n=1}^N a_nn^{-s}$ in the norm
$$\norm{f}_{\sH^p}=\lim_{\sigma\downarrow0}M_p(f,\sigma),$$
where
$$M_p(f,\sigma)=\left(\lim_{T\to\infty}\frac{1}{2T}\int_{-T}^T\abs{f(\sigma+it)}^p\,\dd t\right)^{1/p}$$
are the vertical $p$-means of $f$. The existence of these limits is a consequence of the almost periodicity of Dirichlet polynomials and Hardy's convexity theorem for analytic almost periodic functions. The elements of $\sH^p$ are absolutely convergent Dirichlet series in the half-plane $\bC_{1/2}$ (e.g., \cite{queffelec_diophantine_2020}*{Theorem 6.5.9}), where we write $\bC_\kappa=\{s\in\bC:\Re s>\kappa\}$, and there exist elements of $\sH^p$ that have the line $\Re s=1/2$ as their natural boundary, meaning they cannot be analytically continued to any larger domain (e.g., \cite{queffelec_diophantine_2020}*{Theorem 8.4.4}). This leads to an unexpected interplay between the half-plane $\bC_0$, which we used to define the norm on $\sH^p$, and the half-plane $\bC_{1/2}$, where the functions in $\sH^p$ are defined.

Important to the study of Dirichlet series is their almost periodicity. A Dirichlet series
$$f(s)=\sum_{n=1}^\infty a_nn^{-s}$$
defines an almost periodic function in any half-plane $\bC_\kappa$ where it converges uniformly, and as such it holds that any sequence of vertical translations $f(\,\cdot\,+i\tau)$ with $\tau\in\bR$ has a subsequence that converges uniformly on $\bC_\kappa$ (e.g., \cite{besicovitch_almost_1955}). Since the Dirichlet series of any $f\in\sH^p$ converges absolutely on $\bC_{1/2}$, it also converges uniformly on $\bC_{1/2+\varepsilon}$ for all $\varepsilon>0$, so in particular, $f$ is almost periodic in any such half-plane.

The functions obtained as uniform (on half-planes) limits of vertical translations of $f$ are called the \emph{vertical limit functions} of $f$. A consequence of Kronecker's theorem (e.g., \cite{queffelec_diophantine_2020}*{Theorem 2.2.4}) is that the vertical limit functions of $f$ are precisely the Dirichlet series of the form
$$f_\chi(s)=\sum_{n=1}^\infty a_n\chi(n)n^{-s}$$
where $\chi:\bZ^+\to\bT$ is a completely multiplicative function (e.g., \cite{queffelec_diophantine_2020}*{Proposition 8.4.1}); here $\bT$ denotes the unit circle in the complex plane and $\bZ^+$ the positive integers. We can identify the completely multiplicative functions $\chi:\bZ^+\to\bT$ with the infinite-dimensional torus $\bT^\infty=\bT\times\bT\times\cdots$ by using that $\chi$ is uniquely determined by its values at the primes; that is, $\chi$ corresponds to $(\chi(2),\chi(3),\chi(5),\dots)\in\bT^\infty$, and by the fundamental theorem of arithmetic this identification is bijective. The infinite-dimensional torus $\bT^\infty$ is a compact abelian group, and hence has a unique normalized Haar measure $m_\infty$, which is the product measure of countably many copies of the normalized arclength measure $m$ on $\bT$. This allows us to consider properties that hold for $f_\chi$ for almost every $\chi\in\bT^\infty$.

It is well-known that if $f\in\sH^p$, then $f_\chi$ has an analytic continuation to $\bC_0$ for almost every $\chi\in\bT^\infty$, and there exist several different proofs of this (see, e.g., \cite{helson_compact_1969} for the case $p=2$ and \cite{bayart_hardy_2002}*{Theorem 6} for the general case). In \cite{brevig_carlsons_2025}, Brevig and Kouroupis showed that if $f\in\sH^p$, then for almost every $\chi\in\bT^\infty$ it is also the case that the vertical $p$-mean $M_p(f_\chi,\sigma)$ exists for all $\sigma>0$, defines a decreasing logarithmically convex function in $\sigma$, and can be used to compute the norm of $f$ as
$$\norm{f}_{\sH^p}=\lim_{\sigma\downarrow0}M_p(f_\chi,\sigma).$$
We push this further by showing that the norm of $f$ also satisfies a Littlewood--Paley identity in terms of $f_\chi$ for almost every $\chi\in\bT^\infty$. We write $f(+\infty)$ for the constant term $a_1$ in the Dirichlet series of $f$.

\begin{theorem}[Littlewood--Paley formula for $\sH^p$]\label{theorem:littlewood-paley-hp}
    Let $f\in\sH^p$ with $1\leq p<\infty$. Then
    $$\norm{f}_{\sH^p}^p=\abs{f(+\infty)}^p+\lim_{T\to\infty}\frac{p^2}{2T}\int_{-T}^T\int_0^\infty\abs{f_\chi(\sigma+it)}^{p-2}\abs{f_\chi'(\sigma+it)}^2\sigma\,\dd\sigma\,\dd t$$
    for almost every $\chi\in\bT^\infty$.
\end{theorem}

This provides a significant improvement on previous Littlewood--Paley formulas, and should be compared with \cite{bayart_approximation_2016}*{Theorem 5.1} and \cite{brevig_almost_2025}*{Corollary 1.2, Theorem 1.4}. The result of the former provides a Littlewood--Paley formula in terms of an average over $\bT^\infty$, whereas the result of the latter shows the formula in \cref{theorem:littlewood-paley-hp} under the additional assumption that the Dirichlet series of $f$ converges uniformly on $\bC_\kappa$ for all $\kappa>0$.

In contrast to how the Littlewood--Paley formula in \cite{brevig_almost_2025} was established, we do not need a full Hardy--Stein identity for the vertical $p$-means to prove \cref{theorem:littlewood-paley-hp}; the idea is to integrate after an application of Green's theorem but before establishing a full such identity. This yields a version of the Littlewood--Paley formula with an additional limit taken after the limit in $T$ (see \cref{corollary:weak-littlewood-paley-hp} below), and so to obtain \cref{theorem:littlewood-paley-hp}, we use an ergodic version of the monotone convergence theorem (see \cref{theorem:ergodic-monotone-convergence} below). Although an easy consequence of the ergodic theorem and the monotone convergence theorem, this ergodic version of the monotone convergence theorem proves useful at several points in our arguments, and we believe ergodic versions of the classical integral theorems could provide a useful tool in the further study of $\sH^p$. Due to its independent interest, we do also prove a Hardy--Stein identity for the vertical limit functions of $f$ (see \cref{theorem:hardy-stein-hp} below). In addition to the application of Green's theorem used to establish the Littlewood--Paley formula, the proof will also use that $f$ can be approximated by its Riesz means in a sufficiently strong sense.

Previous Littlewood--Paley formulas have been extensively used in the study of operator theory on $\sH^p$, for example in the study of Volterra operators \cite{brevig_volterra_2019} and in the study of composition operators \cites{brevig_mean_2021,kouroupis_composition_2023,bayart_approximation_2016,bayart_counting_2024}. Motivated by this, we also apply our general scheme to the mean counting function, which plays a central role in the study of compactness of composition operators on $\sH^2$. For an analytic function $f$ in $\bC_0$, its \emph{mean counting function} is defined by
\begin{equation}\label{eq:mean-counting}
\sM(f,\zeta)=\lim_{\sigma_0\downarrow0}\lim_{T\to\infty}\frac{\pi}{T}\sum_{\substack{s\in\sZ(f-\zeta)\\\abs{\Im s}<T\\\Re s>\sigma_0}}(\Re s-\sigma_0)
\end{equation}
for all $\zeta\in\bC\setminus\{f(+\infty)\}$ such that it exists, where we write $\sZ(g)$ for the zero set of $g$, and repeat all terms by the multiplicity of the corresponding zeros. In the context of Dirichlet series, the mean counting function plays the role of the Nevanlinna counting function for analytic functions on the disk. In \cite{brevig_mean_2021}*{Theorem 6.4}, Brevig and Perfekt established the existence of $\sM(f,\zeta)$ when $f\in\sH^p$ and its Dirichlet series converges uniformly on $\bC_\kappa$ for all $\kappa>0$, and furthermore $\sM(f,\zeta)=\sM(f_\chi,\zeta)$ for all $\chi\in\bT^\infty$. They then used the mean counting function to give a complete characterization of the compact composition operators on $\sH^2$ with zero characteristic (see \cite{brevig_mean_2021}*{Theorem 1.4}). It should be stressed that the crucial difference between this setting and our setting is that the functions they consider are bounded and almost periodic on $\bC_\kappa$ for all $\kappa>0$; this plays an important role in their proof of this result.

Motivated by the invariance of the mean counting function under vertical limits, we prove that if $f\in\sH^p$, then there is a set $E\subseteq\bT^\infty$ of full measure such that $\sM(f_\chi,\zeta)$ exists for all $\chi\in E$ and all $\zeta\in\bC\setminus\{f(+\infty)\}$, with its value independent of $\chi\in E$ (see \cref{theorem:existence-of-mean-counting-function} below for the precise statement). As $f$ is not necessarily even defined on all of $\bC_0$, we define $\sM(f,\zeta):=\sM(f_\chi,\zeta)$ for $\chi\in E$ and $\zeta\in\bC\setminus\{f(+\infty)\}$. By employing again our ergodic version of the monotone convergence theorem, we show in particular that one may interchange the limits in $\sM(f,\zeta)$ for almost every $\chi\in\bT^\infty$ (depending on $\zeta$).

\begin{theorem}\label{theorem:mean-counting-limit-interchange}
    Let $f\in\sH^p$ with $1\leq p<\infty$ and let $\zeta\in\bC\setminus\{f(+\infty)\}$. Then
    $$\sM(f,\zeta)=\lim_{T\to\infty}\frac{\pi}{T}\sum_{\substack{s\in\sZ(f_\chi-\zeta)\\\abs{\Im s}<T\\\Re s>0}}\Re s$$
    for almost every $\chi\in\bT^\infty$.
\end{theorem}

This result should be compared to \cite{kouroupis_composition_2023}*{Theorem 4.9} where Kouroupis and Perfekt showed the corresponding statement for the weighted mean counting function under the additional assumption that the Dirichlet series of $f$ converges uniformly on $\bC_\kappa$ for all $\kappa>0$.

The idea behind the proof of the existence of the mean counting function is to first prove that the \emph{Jessen function}
$$\sJ(f_\chi,\sigma)=\lim_{T\to\infty}\frac{1}{2T}\int_{-T}^T\log\abs{f_\chi(\sigma+it)}\,\dd t$$
exists and is decreasing and convex for almost every $\chi\in\bT^\infty$ by adapting an argument of Borchsenius and Jessen \cite{borchsenius_mean_1948} (see \cref{theorem:existence-of-jessen-function} below). This should be compared with the integral
$$\int_{\bT^\infty}\log\abs{f_\chi(\sigma)}\,\dd m_\infty(\chi)$$
for $f\in\sH^p$ and $\sigma>0$ considered by Brevig and Perfekt in \cite{brevig_mean_2021}, which they took as the definition of the Jessen function of $f\in\sH^p$ to deal with the behavior in the strip $0<\Re s\leq1/2$. They furthermore used the ergodic theorem to show that the two definitions agree whenever the Dirichlet series of $f$ converges uniformly (see \cite{brevig_mean_2021}*{Lemma 4.1}), and by a similar application we see that our Jessen function also agrees with theirs. The strength of our result is that we find a set of $\chi\in\bT^\infty$ of full measure for which the Jessen function exists for all $\sigma>0$, whereas if one applies the ergodic theorem directly to their integral, the resulting set depends on $\sigma$.

Once the existence of the Jessen function has been established, we proceed via Littlewood's argument principle in the same way as Brevig and Perfekt used it in \cite{brevig_mean_2021}*{Lemma 6.1} (see also \cite{brevig_almost_2025}*{Theorem 5.3}) to establish a version of Jensen's formula (see \cref{theorem:jensen-formula} below). The existence of the mean counting function $\sM(f_\chi,\zeta)$ in the sense of \cref{eq:mean-counting} is then a consequence of this, and by an ergodic argument \cref{theorem:mean-counting-limit-interchange} then follows. It should be stressed that the set of full measure in \cref{theorem:mean-counting-limit-interchange} depends on $\zeta\in\bC\setminus\{f(+\infty)\}$; however, for the existence of $\sM(f_\chi,\zeta)$ in the sense of \cref{eq:mean-counting}, we can find one such set that works for all choices of $\zeta\in\bC\setminus\{f(+\infty)\}$.

The independence of $\chi$ in the mean counting function is a statement about the average behavior of the distribution of the values of $f$ in $\bC_0$ in a generalized sense, even when $f$ itself cannot necessarily be analytically continued past $\Re s=1/2$. This phenomenon---that almost every vertical limit has a well-controlled value distribution---should be compared with a result of Helson \cite{helson_compact_1969} (see also \cite{hedenmalm_hilbert_1997}*{Corollary 4.7}), stating that for almost every $\chi\in\bT^\infty$, the Riemann hypothesis holds for $\zeta_\chi$, where $\zeta(s)=\sum_{n=1}^\infty n^{-s}$ denotes the Riemann zeta function.

\subsection*{Acknowledgments.} I would like to thank my PhD supervisor, Ole Fredrik Brevig, for his many helpful suggestions and feedback in the writing of this paper.

\subsection*{Organization.} The remainder of this paper is divided into five sections. In \cref{section:preliminaries}, we go over some preliminary material and prove some preliminary lemmas. In particular, we introduce the main tools that will be used throughout the paper: the connection between $\sH^p$ and the Hardy space $H^p(\bT^\infty)$, Riesz means of a Dirichlet series, and a short discussion of almost periodicity. We also introduce in this section what we call the $p$-Carlson set, which will play a prominent role in our results. In \cref{section:ergodic-integral-theorems}, we prove ergodic versions of the monotone and dominated convergence theorems for the Kronecker flow, as well as an ergodic version of Fatou's lemma. Next, \cref{section:zeros-and-jessen} discusses the zero behavior of Dirichlet series by an argument of Borchsenius and Jessen, and finishes with a proof of the existence of the Jessen function for certain analytic continuations of $\sH^p$-functions. In \cref{section:hardy-stein-littlewood-paley}, we establish a Hardy--Stein identity for $\sH^p$ and use this to prove \cref{theorem:littlewood-paley-hp}. Finally, in \cref{section:jensens-formula-and-mean-counting-function} we prove a version of Jensen's formula for the Jessen function and use this to prove the existence of the mean counting function and \cref{theorem:mean-counting-limit-interchange}.

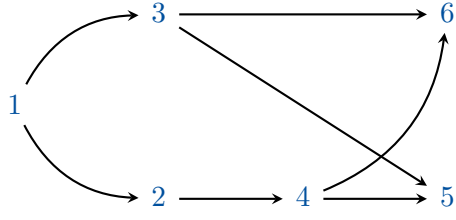
\begin{figure}[ht]
	\centering
	\begin{tikzcd}[row sep=normal,column sep=large]
    & \labelcref{section:ergodic-integral-theorems}\arrow[to=1-4]\arrow[to=3-4] & & \labelcref{section:jensens-formula-and-mean-counting-function} \\
    \labelcref{section:introduction}\arrow[dr, bend right]\arrow[ur, bend left] & & & \\
    & \labelcref{section:preliminaries} \arrow[r] & \labelcref{section:zeros-and-jessen}\arrow[to=1-4,bend right] \arrow[r] & \labelcref{section:hardy-stein-littlewood-paley}
	\end{tikzcd}
	\caption{Dependence between the sections of the paper.}
\end{figure}

\section{Preliminaries}\label{section:preliminaries}

Recall how we identify $\bT^\infty$ with the completely multiplicative functions $\chi:\bZ^+\to\bT$ by their values on the primes. Any such completely multiplicative function uniquely extends to a completely multiplicative function $\chi:\bQ^+\to\bT$, and these functions form the dual group of the positive rationals $\bQ^+$ under the discrete topology. By identifying $\chi\in\bT^\infty$ with its corresponding character $\chi:\bQ^+\to\bT$, we obtain an isomorphism of topological groups, and as such we can compute the Fourier coefficients of a function $f^*\in L^p(\bT^\infty)$ as
$$\widehat{f^*}(q)=\int_{\bT^\infty}f^*(\chi)\overline{\chi(q)}\,\dd m_\infty(\chi)$$
for $q\in\bQ^+$. The Hardy spaces of $\bT^\infty$ are defined as
$$H^p(\bT^\infty)=\{f^*\in L^p(\bT^\infty):\widehat{f^*}(q)=0\text{ for all }q\in\bQ^+\setminus\bZ^+\}.$$
One of the most important tools when studying $\sH^p$ is that it can be isometrically identified with $H^p(\bT^\infty)$ through an idea going back to Harald Bohr (see \cite{bohr_uber_1913_2}). Bohr's idea is to consider the correspondence
$$f(s)=\sum_{n=1}^N a_nn^{-s}\quad\longleftrightarrow\quad f^*(\chi)=\sum_{n=1}^Na_n\chi(n)$$
between Dirichlet polynomials and polynomials on $\bT^\infty$, which is a bijective linear correspondence. Consider the \emph{Kronecker flow}
$$\fp^{-it}=(2^{-it},3^{-it},5^{-it},\dots)$$
for $t\in\bR$, and observe that, plainly, $f(it)=f^*(\fp^{-it})$. The family $\{\fp^{-it}\}_{t\in\bR}$ defines an ergodic flow on $\bT^\infty$ with respect to $m_\infty$, and so by the ergodic theorem (e.g., \cite{queffelec_diophantine_2020}*{Theorem 6.5.1}) it holds that
\begin{equation}\label{eq:ergodic-theorem-polynomials}
    \norm{f}_{\sH^p}^p=\lim_{T\to\infty}\frac{1}{2T}\int_{-T}^T\abs{f(it)}^p\,\dd t=\int_{\bT^\infty}\abs{f^*(\chi)}^p\,\dd m_\infty(\chi)=\norm{f^*}_{L^p(\bT^\infty)}^p.
\end{equation}
In particular, the correspondence $f\mapsto f^*$ is isometric, and so as the polynomials are dense in $H^p(\bT^\infty)$, this extends to an isometric isomorphism between $\sH^p$ and $H^p(\bT^\infty)$, often called the \emph{Bohr lift}. We will write $f^*$ for the Bohr lift of $f\in\sH^p$.

The application of the ergodic theorem in \cref{eq:ergodic-theorem-polynomials} relies crucially on the fact that $f^*$ is continuous on $\bT^\infty$ (indeed it is a polynomial). For general $f^*\in L^p(\bT^\infty)$, we can only infer from the ergodic theorem that
$$\norm{f^*}_{L^p(\bT^\infty)}^p=\lim_{T\to\infty}\frac{1}{2T}\int_{-T}^T\abs{f^*(\chi\fp^{-it})}^p\,\dd t$$
for almost every $\chi\in\bT^\infty$. Indeed it should be noted that even if $f$ is a Dirichlet series with bounded analytic continuation to $\bC_0$, the first limit in \cref{eq:ergodic-theorem-polynomials} may not even exist, and even when it does it may not equal the $\sH^p$-norm of $f$ (see \cite{saksman_integral_2009}*{Theorem 1}). Brevig and Kouroupis \cite{brevig_carlsons_2025} first observed that the set of $\chi$ for which the above equality holds plays an important role in the study of $\sH^p$ and gave the following definition.

\begin{definition}
    Let $f\in\sH^p$ with $1\leq p<\infty$. The set of all $\chi\in\bT^\infty$ for which the equality
    $$\norm{f}_{\sH^p}^p=\lim_{T\to\infty}\frac{1}{2T}\int_{-T}^T\abs{f^*(\chi\fp^{-it})}^p\,\dd t$$
    holds is called the \emph{ergodic set} of $f$, and we denote it by $\cE_p(f)$.
\end{definition}

The ergodic theorem ensures that if $f\in\sH^p$, then $\cE_p(f)$ has full measure, and if $f^*$ is continuous on $\bT^\infty$, or equivalently, if the Dirichlet series of $f$ converges pointwise and is uniformly continuous on $\bC_0$ (e.g., \cite{aron_dirichlet_2017}*{Theorem 2.3}), then $\cE_p(f)=\bT^\infty$.

In our scheme, the most important property of the ergodic set is the fact that if $\chi\in\cE_p(f)$, then
\begin{equation}\label{eq:carlson-condition}
\sup_{\sigma>0}\sup_{T\geq1}\frac{1}{2T}\int_{-T}^T\abs{f_\chi(\sigma+it)}^p\,\dd t<\infty
\end{equation}
by \cite{brevig_carlsons_2025}*{Theorem 2}. It is precisely this condition that will underpin our arguments. Intuitively, one should see \cref{eq:carlson-condition} as a replacement for boundedness in the strip $0<\Re\leq1/2$ which ensures sufficiently strong growth control for the function to be well-behaved with respect to vertical means. In general, however, the ergodic set is rather poorly behaved. For example, it is not necessarily the case that $\cE_p(f)\cap\cE_p(g)$ is contained in $\cE_p(f+g)$ even though it certainly is the case that if \cref{eq:carlson-condition} holds for $f$ and $g$, then it also holds for $f+g$. As such, we give the following definition.

\begin{definition}
    For a somewhere convergent Dirichlet series $f$ and $1\leq p<\infty$, we define the \emph{$p$-Carlson set} of $f$ as the set $\sC_p(f)$ of all $\chi\in\bT^\infty$ such that $f_\chi$ has an analytic continuation to $\bC_0$ satisfying
    $$\sup_{\sigma>0}\sup_{T\geq1}\frac{1}{2T}\int_{-T}^T\abs{f_\chi(\sigma+it)}^p\,\dd t<\infty.$$
\end{definition}

Considering the $p$-Carlson set instead of the ergodic set is rather natural in the sense that it captures precisely the behavior we want: a sufficiently strong replacement for boundedness in $0<\Re s\leq1/2$. As mentioned implicitly before, it holds that $\cE_p(f)\subseteq\sC_p(f)$, but in contrast to the ergodic set, the $p$-Carlson set enjoys better arithmetic properties. For example, it is clearly the case that
$$\sC_p(f)\cap\sC_p(g)\subseteq\sC_p(f+g)$$
and
$$\sC_{p_2}(f)\subseteq\sC_{p_1}(f)$$
whenever $p_1\leq p_2$. Furthermore, by combining \cite{brevig_carlsons_2025}*{Theorems 1-2}, we obtain the following characterization of $\sH^p$ in terms of $p$-Carlson sets.

\begin{theorem}\label{theorem:characterization-of-hp-carlson-set}
    Let $f$ be a somewhere convergent Dirichlet series and let $1\leq p<\infty$. The following are equivalent:
    \begin{enumerate}[(i)]
        \item $f\in\sH^p$,
        \item $\sC_p(f)$ is non-empty,
        \item $\sC_p(f)$ has full measure.
    \end{enumerate}
\end{theorem}

It is worth noting, however, that the ergodic set generally gives $f_\chi$ better boundary behavior than the $p$-Carlson set does. Indeed if $\chi\in\cE_p(f)$, then by \cite{brevig_carlsons_2025}*{Corollary 3} it is the case that
$$f_\chi(it):=\lim_{\sigma\downarrow0}f_\chi(\sigma+it)$$
exists for all $t\in\bR$ and can be used to compute the $\sH^p$-norm of $f$ as
\begin{equation}\label{eq:norm-computation-boundary}
\norm{f}_{\sH^p}=\left(\lim_{T\to\infty}\frac{1}{2T}\int_{-T}^T\abs{f_\chi(it)}^p\,\dd t\right)^{1/p}.
\end{equation}
This does not necessarily hold for $\chi\in\sC_p(f)$: for example, if $f$ has a bounded analytic continuation to $\bC_0$, then $\sC_p(f)=\bT^\infty$, but \cite{saksman_integral_2009}*{Theorem 1} of Saksman and Seip tells us that there exists such an $f$ for which \cref{eq:norm-computation-boundary} does not hold with $\chi=(1,1,1,\dots)$.

As we will be concerned with vertical limits by elements of the $p$-Carlson set, we shall start with some basic properties of Dirichlet series $f$ satisfying
$$\sup_{\sigma>0}\sup_{T\geq1}\frac{1}{2T}\int_{-T}^T\abs{f(\sigma+it)}^p\,\dd t<\infty.$$
The first such property will concern the order of such functions. Recall that if $f$ is an analytic function in a half-plane $\bC_\kappa$, then its \emph{order} on $\bC_\kappa$, denoted by $\mu_f(\kappa)$, is the infimum over all real numbers $\mu$ such that
$$f(\sigma+it)=O(\abs{t}^\mu)$$
uniformly for $\sigma>\kappa$ as $\abs{t}\to\infty$. The function $\mu_f$ is convex on $(\kappa,\infty)$, and in particular also continuous, and if $f$ is an absolutely convergent Dirichlet series on $\bC_\kappa$, then $\mu_f(\kappa)=0$ (see, e.g., \cite{hardy_general_1964}*{Theorem 15} and \cite{titchmarsh_theory_1958}*{Section 9.41}).

\begin{lemma}\label{lemma:order-strict-inequality}
    Let $1\leq p<\infty$ and let $f$ be a somewhere convergent Dirichlet series with analytic continuation to $\bC_0$ satisfying
    $$\sup_{\sigma>0}\sup_{T\geq1}\frac{1}{2T}\int_{-T}^T\abs{f(\sigma+it)}^p\,\dd t<\infty.$$
    Then $\mu_f(\kappa)<1/p$ for all $\kappa>0$.
\end{lemma}

\begin{proof}
    It follows from \cite{brevig_carlsons_2025}*{Lemma 7} that $\mu_f(\kappa)\leq1/p$ for all $\kappa>0$. The result now follows from the convexity of $\mu_f$ and the fact that $f$ is somewhere absolutely convergent (e.g., \cite{titchmarsh_theory_1958}*{Section 9.13}).
\end{proof}

This control of the order of $f$ will play a prominent role in our arguments. We also need order control on the derivative of $f$, and as such we prove that $f'$ satisfies a similar growth condition.

\begin{lemma}\label{lemma:carlson-derivative-inequality}
    Let $1\leq p<\infty$ and let $f$ be an analytic function in $\bC_0$. Then
    $$\sup_{\sigma>\kappa}\sup_{T\geq1}\frac{1}{2T}\int_{-T}^T\abs{f'(\sigma+it)}^p\,\dd t\leq\frac{1+\kappa}{\kappa^p}\sup_{\sigma>0}\sup_{T\geq1}\frac{1}{2T}\int_{-T}^T\abs{f(\sigma+it)}^p\,\dd t$$
    for all $\kappa>0$.
\end{lemma}

\begin{proof}
    Fix $\kappa>0$ and let $\sigma>\kappa$ and $T\geq 1$. Set
    $$C=\sup_{\sigma>0}\sup_{T\geq1}\frac{1}{2T}\int_{-T}^T\abs{f(\sigma+it)}^p\,\dd t$$
    and use Cauchy's integral formula together with Minkowski's integral inequality to estimate
    \begin{align*}
    \Biggl(\frac{1}{2T}\int_{-T}^T\abs{f'(\sigma+it)}^p&\,\dd t\Biggr)^{1/p} \\
        &=\left(\frac{1}{2T}\int_{-T}^T\Abs{\frac{1}{2\pi}\int_0^{2\pi}\frac{f(\sigma+it+\kappa e^{i\theta})}{\kappa e^{i\theta}}\,\dd\theta}^p\,\dd t\right)^{1/p} \\
        &\leq\frac{1}{2\pi\kappa}\int_0^{2\pi}\left(\frac{1}{2T}\int_{-T}^T\abs{f(\sigma+it+\kappa e^{i\theta})}^p\,\dd t\right)^{1/p}\,\dd\theta \\
        &\leq\frac{1}{2\pi\kappa}\int_0^{2\pi}\left(\frac{1+\kappa }{2(T+\kappa )}\int_{-T-\kappa}^{T+\kappa}\abs{f(\sigma+\kappa\cos\theta+i\tau)}^p\,\dd\tau\right)^{1/p}\,\dd\theta \\
        &\leq C^{1/p}\frac{(1+\kappa)^{1/p}}{\kappa }.
    \end{align*}
    The result now follows by raising both sides to the power of $p$ and taking suprema.
\end{proof}

A consequence of this lemma is also that if $f\in\sH^p$, then any horizontal translation $s\mapsto f'(s+\kappa)$ with $\kappa>0$ of the derivative of $f$ is also in $\sH^p$.

Our next tool is given by the Riesz means of a Dirichlet series. If $f(s)=\sum_{n=1}^\infty a_nn^{-s}$ is a Dirichlet series, then its \emph{Riesz means} are defined by
$$R_N^k f(s)=\sum_{n=1}^N a_n\left(1-\frac{\log n}{\log N}\right)^kn^{-s}$$
for $k>0$ and $N\geq2$. Their utility comes from the fact that they are Dirichlet polynomials---so in particular they are almost periodic functions on any half-plane---that approximate $f$ in a very strong sense. The starting point for our arguments is the following lemma, which can easily be deduced from the proofs of \cite{brevig_carlsons_2025}*{Theorems 1 and 9}.

\begin{lemma}\label{lemma:riesz-means-convergence}
    Let $1\leq p<\infty$ and let $f$ be a somewhere convergent Dirichlet series with analytic continuation to $\bC_0$ satisfying
    \begin{equation}\label{eq:carlson-in-derivative-lemma}
        \sup_{\sigma>0}\sup_{T\geq1}\frac{1}{2T}\int_{-T}^T\abs{f(\sigma+it)}^p\,\dd t<\infty.
    \end{equation}
    Then, for $k>3$, the following statements hold:
    \begin{enumerate}[(a), font=\normalfont]
        \item $R_N^kf\to f$ uniformly on any compact subset of $\bC_0$ as $N\to\infty$.
        \item If the Dirichlet series of $f$ converges uniformly on $\bC_\kappa$, then $R_N^kf\to f$ uniformly on $\bC_\kappa$ as $N\to\infty$.
        \item For all $\kappa>0$,
        $$\lim_{N\to\infty}\sup_{\sigma>\kappa}\sup_{T\geq1}\frac{1}{2T}\int_{-T}^T\abs{R_N^kf(\sigma+it)-f(\sigma+it)}^p\,\dd t=0.$$
    \end{enumerate}
\end{lemma}

We remark here that $(R_N^k f)'=R_N^k(f')$, which is immediate from the definition, and as such we generally only write $R_N^kf'$ for this.

Our final tool is almost periodicity. We call a complex-valued function $f$, defined on a half-plane $\bC_\kappa$, \emph{almost periodic} if it is uniformly continuous, bounded, and for all $\varepsilon>0$, the set of all $\tau>0$ satisfying
$$\abs{f(s+i\tau)-f(s)}\leq\varepsilon$$
for all $s\in\bC_\kappa$ is relatively dense. A bounded uniformly continuous function $f$ is almost periodic if and only if the set of all vertical translations $\{f(\,\cdot\,+i\tau):\tau\in\bR\}$ is precompact in the uniform norm on $\bC_\kappa$. The most important example of almost periodic functions in our setting is uniformly convergent Dirichlet series, and every somewhere convergent Dirichlet series is uniformly convergent on $\bC_\kappa$ for some sufficiently large $\kappa$ (e.g., \cite{queffelec_diophantine_2020}*{Section 4.2}). We will also need that if $f$ is almost periodic on $\bC_\kappa$ and $\psi$ is uniformly continuous on $f(\bC_\kappa)$, then the composition $\psi\circ f$ is almost periodic on $\bC_\kappa$. There are a few more things we need to know about almost periodic functions. The first is that if $f$ is almost periodic in $\bC_\kappa$, then the mean value
$$M(f,\sigma)=\lim_{T\to\infty}\frac{1}{2T}\int_{-T}^T f(\sigma+it)\,\dd t$$
exists uniformly for all $\sigma>\kappa$, and defines a bounded and uniformly continuous function in $\sigma$. Next, if $f$ is an analytic almost periodic function, then $M(f,\sigma)$ is constant, and we denote its value by $f(+\infty)$. It holds that
$$\lim_{\kappa\to\infty}\sup_{s\in\bC_\kappa}\abs{f(s)-f(+\infty)}=0,$$
and in the special case when $f$ is a Dirichlet series, we have that $f(+\infty)$ is the constant term in the series. For more details, see, for example, \cite{besicovitch_almost_1955}.

\section{Ergodic integral theorems for the Kronecker flow}\label{section:ergodic-integral-theorems}

In this section, we will deduce from the ergodic theorem versions of the monotone and dominated convergence theorem, as well as Fatou's lemma. The ergodic version of the monotone convergence theorem will in particular play a role in our later arguments when we wish to interchange limits and vertical means. To prove it we will rely on the following slight extension of the ergodic theorem for the Kronecker flow in the case of non-negative functions.

\begin{lemma}\label{lemma:non-negative-ergodic-theorem}
    If $F:\bT^\infty\to[0,\infty]$ is measurable, then
    $$\int_{\bT^\infty}F(\chi)\,\dd m_\infty(\chi)=\lim_{T\to\infty}\frac{1}{2T}\int_{-T}^T F(\chi'\fp^{-it})\,\dd t$$
    for almost every $\chi'\in\bT^\infty$.
\end{lemma}

\begin{proof}
    If $\int_{\bT^\infty}F(\chi)\,\dd m_\infty(\chi)<\infty$, then this follows immediately from the ergodic theorem for the Kronecker flow, so suppose $\int_{\bT^\infty}F(\chi)\,\dd m_\infty(\chi)=\infty$. Define, for $n\in\bZ^+$ and $\chi\in\bT^\infty$,
    $$F_n(\chi)=\begin{cases}
        F(\chi),&F(\chi)\leq n,\\
        0,&F(\chi)>n.
    \end{cases}$$
    Then $F_n\in L^1(\bT^\infty)$, so by applying the ergodic theorem to $F_n$ we can find a set $E_n\subseteq\bT^\infty$ of full measure such that
    $$\int_{\bT^\infty}F_n(\chi)\,\dd m_\infty(\chi)=\lim_{T\to\infty}\frac{1}{2T}\int_{-T}^TF_n(\chi'\fp^{-it})\,\dd t$$
    for all $\chi'\in E_n$. Set $E=\bigcap_{n\in\bZ^+}E_n$. Then, by monotonicity, we have that
    $$\liminf_{T\to\infty}\frac{1}{2T}\int_{-T}^T F(\chi'\fp^{-it})\,\dd t\geq\lim_{T\to\infty}\frac{1}{2T}\int_{-T}^TF_n(\chi'\fp^{-it})\,\dd t=\int_{\bT^\infty}F_n(\chi)\,\dd m_\infty(\chi)$$
    for all $\chi'\in E$ and all $n\in\bZ^+$. The result now follows by letting $n\to\infty$ in the above inequality and using the monotone convergence theorem.
\end{proof}

From this we readily deduce our ergodic version of the monotone convergence theorem.

\begin{theorem}[Ergodic monotone convergence for the Kronecker flow]\label{theorem:ergodic-monotone-convergence}
    Let $\{F_n\}_{n\in\bZ^+}$ be a non-decreasing sequence of non-negative measurable functions on $\bT^\infty$, and set
    $$F(\chi)=\lim_{n\to\infty}F_n(\chi)$$
    for $\chi\in\bT^\infty$. Then
    $$\lim_{n\to\infty}\lim_{T\to\infty}\frac{1}{2T}\int_{-T}^TF_n(\chi\fp^{-it})\,\dd t=\lim_{T\to\infty}\frac{1}{2T}\int_{-T}^TF(\chi\fp^{-it})\,\dd t$$
    for almost every $\chi\in\bT^\infty$.
\end{theorem}

\begin{proof}
    For each $n\in\bZ^+$, use \cref{lemma:non-negative-ergodic-theorem} to find a set $E_n\subseteq\bT^\infty$ of full measure such that
    \begin{equation}\label{eq:ergodic-application-fn-monotone}
    \lim_{T\to\infty}\frac{1}{2T}\int_{-T}^TF_n(\chi'\fp^{-it})\,\dd t=\int_{\bT^\infty}F_n(\chi)\,\dd m_\infty(\chi)
    \end{equation}
    for all $\chi\in E_n$. Set $E=\bigcap_{n\in\bZ^+}E_n$. Then \cref{eq:ergodic-application-fn-monotone} holds for all $\chi'\in E$ and all $n\in\bZ^+$, so by letting $n\to\infty$ and applying the monotone convergence theorem, it follows that
    $$\lim_{n\to\infty}\lim_{T\to\infty}\frac{1}{2T}\int_{-T}^TF_n(\chi'\fp^{-it})\,\dd t=\int_{\bT^\infty}F(\chi)\,\dd m_\infty(\chi).$$
    The result now follows by applying \cref{lemma:non-negative-ergodic-theorem} again and intersecting the resulting set with $E$.
\end{proof}

As an immediate consequence of the ergodic monotone convergence theorem, we obtain also the following ergodic version of Fatou's lemma; the proof is standard.

\begin{theorem}[Ergodic Fatou's lemma for the Kronecker flow]
    Let $\{F_n\}_{n\in\bZ^+}$ be a sequence of non-negative measurable functions on $\bT^\infty$. Then
    $$\lim_{T\to\infty}\frac{1}{2T}\int_{-T}^T\liminf_{n\to\infty}F_n(\chi\fp^{-it})\,\dd t\leq\liminf_{n\to\infty}\lim_{T\to\infty}\frac{1}{2T}\int_{-T}^TF_n(\chi\fp^{-it})\,\dd t$$
    for almost every $\chi\in\bT^\infty$.
\end{theorem}

By arguing the same way as in the proof of \cref{theorem:ergodic-monotone-convergence}, replacing the applications of the monotone convergence theorem with the dominated convergence theorem, one readily also proves the following ergodic version of the dominated convergence theorem. Although we shall not need it, we include it as it might be of interest to the reader.

\begin{theorem}[Ergodic dominated convergence for the Kronecker flow]
    Let $\{F_n\}_{n\in\bZ^+}$ be a sequence of complex-valued measurable functions on $\bT^\infty$, and suppose that the sequence converges almost everywhere to a function $F$. Suppose also that there exists a function $G\in L^1(\bT^\infty)$ such that
    $$\abs{F_n(\chi)}\leq\abs{G(\chi)}$$
    for almost every $\chi\in\bT^\infty$. Then $F$ and $F_n$ are in $L^1(\bT^\infty)$ for all $n\in\bZ^+$, and
    $$\lim_{n\to\infty}\lim_{T\to\infty}\frac{1}{2T}\int_{-T}^TF_n(\chi\fp^{-it})\,\dd t=\lim_{T\to\infty}\frac{1}{2T}\int_{-T}^TF(\chi\fp^{-it})\,\dd t$$
    for almost every $\chi\in\bT^\infty$.
\end{theorem}

\begin{remark}
    One may clearly replace $\bT^\infty$ with any probability space and the Kronecker flow with any ergodic semigroup of measure-preserving transformations on this space in all three of the above theorems by applying the Birkhoff ergodic theorem in this setting (e.g., \cite{queffelec_diophantine_2020}*{Theorem 2.1.12}).
\end{remark}

\section{Zeros of Dirichlet series and the Jessen function}\label{section:zeros-and-jessen}

In this section we will be concerned with the zeros of somewhere convergent Dirichlet series $f$ with analytic continuation to $\bC_0$ satisfying
\begin{equation}\label{eq:carlson-condition-zero-section}
    \sup_{\sigma>0}\sup_{T\geq1}\frac{1}{2T}\int_{-T}^T\abs{f(\sigma+it)}^p\,\dd t<\infty
\end{equation}
for some $1\leq p<\infty$. In any half-plane $\bC_\kappa$ where $f$ is almost periodic, the properties of the zeros of $f$ were studied originally by Jessen in \cite{jessen_uber_1933}, and later and in more depth by Jessen and Tornehave in \cite{jessen_mean_1945}. Our main concern is the original work of Jessen, who showed that if $f$ is almost periodic on $\bC_\kappa$, then the \emph{Jessen function}
$$\sJ(f,\sigma)=\lim_{T\to\infty}\frac{1}{2T}\int_{-T}^T\log\abs{f(\sigma+it)}\,\dd t$$
exists and is finite for any $\sigma>\kappa$. In our setting, we know that $f$ is almost periodic in $\bC_\kappa$ for all $\kappa>1/2$. In view of the results of Jessen, our main concern is thus the behavior of the zero set of $f$ in the strip $0<\Re s\leq1/2$.

Borchsenius and Jessen \cite{borchsenius_mean_1948}, building on this work, established the existence of the Jessen function in a much more general setting. Instead of assuming almost periodicity, they assumed the existence of a sequence of almost periodic functions that approximates the function in a sufficiently strong sense. Combining their result with \cref{lemma:riesz-means-convergence}, it follows that the Jessen function $\sJ(f,\sigma)$ exists for $0<\sigma\leq1/2$ under the assumption \cref{eq:carlson-condition-zero-section}. We shall present their argument with some simplifications and adaptations to our setting; however, we make it clear that some variant of most results in this section can be found in \cite{borchsenius_mean_1948}. The goal of this is twofold: to give a more complete treatment of the Jessen function and the arguments used to establish its existence; and to demonstrate the techniques used, which we believe will be relevant for further understanding the zero sets of $\sH^p$-functions. We shall use some of the techniques of Borchsenius and Jessen in our later arguments, and so we deem this a natural place to introduce them.

The main goal of the current section is thus to prove the following theorem.

\begin{theorem}\label{theorem:existence-of-jessen-function}
    Let $1\leq p<\infty$ and let $f$ be a somewhere convergent Dirichlet series that is not identically zero with analytic continuation to $\bC_0$ satisfying
    $$\sup_{\sigma>0}\sup_{T\geq1}\frac{1}{2T}\int_{-T}^T\abs{f(\sigma+it)}^p\,\dd t<\infty.$$
    Then the Jessen function
    $$\sJ(f,\sigma)=\lim_{T\to\infty}\frac{1}{2T}\int_{-T}^T\log\abs{f(\sigma+it)}\,\dd t$$
    exists and is finite for all $\sigma>0$, and the limit converges uniformly in $\sigma$ on $(\alpha,\beta)$ for all $0<\alpha<\beta$. Furthermore, the function $\sigma\mapsto\sJ(f,\sigma)$ is decreasing and convex on $(0,\infty)$, and for all $k>3$ it holds that
    $$\lim_{N\to\infty}\sJ(R_N^kf,\sigma)=\sJ(f,\sigma)$$
    uniformly in $\sigma$ on $(\alpha,\beta)$ for all $0<\alpha<\beta$. Finally, if $f(+\infty)\neq0$, then the limits are uniform in $\sigma$ on $(\kappa,\infty)$ for all $\kappa>0$.
\end{theorem}

Let us start by establishing some notation. As mentioned in the introduction, we write $\sZ(f)$ for the zero set of $f$, and whenever we take a sum or product over $\sZ(f)$ (or $\sZ(f)\cap\Omega$ for some set $\Omega$) we repeat each term according to the multiplicity of the corresponding zero. If $\Omega\subseteq\bC$, then we will write $N_f(\Omega)$ for the number of zeros of $f$ in $\Omega$ counting multiplicity; that is,
$$N_f(\Omega)=\sum_{s\in\sZ(f)\cap\Omega}1.$$
We will also write
$$\norm{f}_\Omega=\sup_{s\in\Omega}\abs{f(s)}.$$

The underlying principle behind the techniques of Borchsenius and Jessen is to estimate the zero behavior of $f$ in a given domain $\Omega_1$ by the quantity $\norm{f}_{\Omega_2}$ for a bigger domain $\Omega_2$, in a way where this control can be done uniformly over vertical translations.

As a small preliminary lemma, we need the following simple consequence of the Riemann mapping theorem.

\begin{lemma}\label{lemma:riemann-mapping-lemma}
    Let $\Omega_1,\Omega_2$ be bounded simply connected domains in $\bC$ with $\overline\Omega_1\subseteq\Omega_2$. Then there exists a family $\{\varphi_w\}_{w\in\overline\Omega_1}$ of Riemann maps $\varphi_w:\Omega_2\to\bD$ and a $\rho\in(0,1)$ such that $\varphi_w(w)=0$ and $\varphi_w(\Omega_1)\subseteq\rho\bD$ for all $w\in\overline\Omega_1$.
\end{lemma}

\begin{proof}
    Fix an arbitrary Riemann map $\varphi:\Omega_2\to\bD$. For $\alpha\in\bD$ consider the disk automorphism $\psi_\alpha:\bD\to\bD$ interchanging $\alpha$ and $0$, i.e.,
    \begin{equation}\label{eq:disk-automorphism}
    \psi_\alpha(z)=\frac{\alpha-z}{1-\overline\alpha z}.
    \end{equation}
    Set $\varphi_w=\psi_{\varphi(w)}\circ\varphi:\Omega_2\to\bD$ for $w\in\Omega_2$. Clearly $\varphi_w$ is a Riemann map with $\varphi_w(w)=0$, and furthermore it is clear from \cref{eq:disk-automorphism} that $(z,w)\mapsto\varphi_w(z)$ is continuous on $\Omega_2\times\Omega_2$. In particular, as $\overline\Omega_1\times\overline\Omega_1$ is compact, the result follows by taking $\rho=\max_{z,w\in\overline\Omega_1}\abs{\varphi_w(z)}$
\end{proof}

The first part of this general scheme is to establish some control on the zero behavior of $f$ in some domain by its maximum on some larger domain in a way that we will be able to do uniformly over vertical translations. The main tool for this is the following consequence of Jensen's formula.

\begin{lemma}\label{lemma:zero-count-bound}
    Let $\Omega_1,\Omega_2$ be bounded simply connected domains in $\bC$ with $\overline\Omega_1\subseteq\Omega_2$, let $f$ be an analytic function on $\Omega_2$, and let $m>0$ be such that $\norm{f}_{\Omega_1}\geq m$. Then there exists a constant $C>0$, depending only on $\Omega_1,\Omega_2$ and $m$, such that
    $$N_f(\Omega_1)\leq C\log(\norm{f}_{\Omega_2}+1).$$
\end{lemma}

\begin{proof}
    Let $\{\varphi_w\}_{w\in\overline\Omega_1}$ and $\rho\in(0,1)$ be as in \cref{lemma:riemann-mapping-lemma}, and let $w\in\overline\Omega_1$ be such that $\abs{f(w)}\geq m$. Then, applying Jensen's formula (e.g., \cite{conway_functions_1978}*{Chapter XI}) to $f\circ\varphi_w^{-1}$ with $r\in(\rho,1)$, we can estimate
    \begin{multline*}
    \log m
        \leq\log\abs{f(\varphi_w^{-1}(0))}
        =\sum_{z\in\sZ(f\circ\varphi_w^{-1})\cap\rho\bD}\log\frac{\abs{z}}{r}+\int_\bT\log\abs{f(\varphi_w^{-1}(r\zeta))}\,\dd m(\zeta) \\
        \leq\sum_{s\in\sZ(f)\cap\Omega_1}\log\frac{\abs{\varphi_w(s)}}{r}+\log\norm{f}_{\Omega_2} 
        \leq N_f(\Omega_1)\log\frac{\rho}{r}+\log\norm{f}_{\Omega_2}.
    \end{multline*}
    The result now follows after letting $r\uparrow1$, rearranging, and estimating $\norm{f}_{\Omega_2}/m\leq(\norm{f}_{\Omega_2}+1)^{1/\min\{m,1\}}$.
\end{proof}

The first application of the above lemma comes in the form of controlling the modulus of $f$ from below by its zeros.

\begin{lemma}\label{lemma:zero-product-estimate}
    Let $\Omega_1,\Omega_2$ be bounded simply connected domains in $\bC$ with $\overline\Omega_1\subseteq\Omega_2$, let $f$ be a bounded analytic function on $\Omega_2$, and let $m>0$ be such that $\norm{f}_{\Omega_1}\geq m$. Then there exists a constant $\alpha>0$, depending only on $\Omega_1,\Omega_2$ and $m$, such that
    $$\abs{f(s)}\geq\frac{1}{(\norm{f}_{\Omega_2}+1)^\alpha}\prod_{s'\in\sZ(f)\cap\Omega_1}\abs{s-s'}$$
    for all $s\in\Omega_1$.
\end{lemma}

\begin{proof}
    Set
    $$g(s)=f(s)\prod_{s'\in\sZ(f)\cap\Omega_1}\frac{1}{s-s'}$$
    for $s\in\Omega_2$, and note that $g$ is a bounded analytic function on $\Omega_2$. By an application of the maximum modulus principle it satisfies
    $$\norm{g}_{\Omega_2}\leq\frac{\norm{f}_{\Omega_2}}{\dist(\Omega_1,\partial\Omega_2)^{N_f(\Omega_1)}}.$$
    Let now $\{\varphi_w\}_{w\in\overline\Omega_1}$ and $\rho\in(0,1)$ be as in \cref{lemma:riemann-mapping-lemma} and let $w\in\overline\Omega_1$ be such that $\abs{f(w)}\geq m$. Since $g$ is zero-free in $\Omega_1$, we can let $\log g$ denote a branch of its logarithm in $\Omega_1$. Apply the Borel--Carathéodory inequality (see, e.g., \cite{titchmarsh_theory_1958}*{Chapter V}) to the function $\log g\circ\varphi_w^{-1}-\log g(w)$ and bound its real part by its modulus to estimate
    \begin{multline*}
    \log\abs{g(s)}-\log\abs{g(w)}
        \leq\frac{2\rho}{1-\rho}(\log\norm{g\circ\varphi_w^{-1}}_\bD-\log\abs{g(w)}) \\
        \leq\frac{2\rho}{1-\rho}\left(\log\frac{\norm{f}_{\Omega_2}}{\dist(\Omega_1,\partial\Omega_2)^{N_f(\Omega_1)}}-\log\frac{m}{\diam(\Omega_1)^{N_f(\Omega_1)}}\right)
    \end{multline*}
    for $s\in\Omega_1$. From this we obtain the estimate
    \begin{multline*}
    \log\abs{g(s)}\geq N_f(\Omega_1)\left(\frac{2\rho}{1-\rho}\log\dist(\Omega_1,\partial\Omega_2)-\frac{1+\rho}{1-\rho}\log\diam(\Omega_1)\right) \\
    +\frac{1+\rho}{1-\rho}\log m-\frac{2\rho}{1-\rho}\log\norm{f}_{\Omega_2}
    \end{multline*}
    for $s\in\Omega_1$. From here, the first term can be estimated from below using \cref{lemma:zero-count-bound} and the rest can be estimated from below using that $\norm{f}_{\Omega_2}\geq m$, from which the result readily follows; we leave the details to the reader.
\end{proof}

The above two lemmas now suffice to prove the main estimate needed for the existence of the Jessen function. We will write
$$\abs{z}_\delta=\max\{\abs{z},\delta\}$$
for $\delta>0$ and $z\in\bC$.

\begin{lemma}\label{lemma:log-estimate-c-family}
    Let $\Omega$ be a bounded simply connected domain in $\bC$, let $\alpha<\beta$ and $T_1<T_2$ be such that $R=[\alpha,\beta]\times[T_1,T_2]\subseteq\Omega$, let $f$ be an analytic function on $\Omega$, and let $m>0$ be such that $\norm{f}_R\geq m$. Then there exists a family $\{C_\delta\}_{\delta>0}$ of positive numbers, depending only on $\Omega,R$ and $m$, such that
    $$\int_{T_1}^{T_2}(\log\abs{f(\sigma+it)}_\delta-\log\abs{f(\sigma+it)})\,\dd t\leq C_\delta\log^2(\norm{f}_\Omega+1)$$
    for all $\sigma\in[\alpha,\beta]$ and all $\delta>0$, and such that
    $$\lim_{\delta\downarrow0}C_\delta=0.$$
\end{lemma}

\begin{proof}
    Let $\varepsilon>0$. The goal is to show that there exists a $\delta_0>0$, depending only on $\Omega,R$ and $m$, such that if $\delta\in(0,\delta_0]$, then
    \begin{equation}\label{eq:log-integral-epsilon}
        \int_{T_1}^{T_2}(\log\abs{f(\sigma+it)}_\delta-\log\abs{f(\sigma+it)})\,\dd t\leq\varepsilon\log^2(\norm{f}_\Omega+1)
    \end{equation}
    for all $\sigma\in[\alpha,\beta]$. Let $J_\delta(\sigma)$ denote the left-hand side in \cref{eq:log-integral-epsilon}. Use first \cref{lemma:zero-count-bound,lemma:zero-product-estimate} to find constants $C,\alpha>0$, depending only on $\Omega,R$ and $m$, such that
    \begin{equation}\label{eq:zeros-and-product-estimates}
    N_f(R)\leq C\log(\norm{f}_\Omega+1),\quad\abs{f(s)}\geq\frac{1}{(\norm{f}_\Omega+1)^\alpha}\prod_{s'\in\sZ(f)\cap R}\abs{s-s'}
    \end{equation}
    for all $s\in R$. Let $\delta\in(0,\delta_0]$. Denote by $I_{\sigma,\delta}$ the set of all $t\in[T_1,T_2]$ such that $\abs{f(\sigma+it)}\leq\delta$, and estimate
    \begin{equation}\label{eq:j-delta-first-estimate}
    J_\delta(\sigma)\leq\abs{I_{\sigma,\delta}}(\log\delta+\alpha\log(\norm{f}_\Omega+1))+\sum_{s'\in\sZ(f)\cap R}\int_{I_{\sigma,\delta}}\log\frac{1}{\abs{t-\Im s'}}\,\dd t
    \end{equation}
    where we use $\abs{\,\cdot\,}$ to denote Lebesgue measure. First, using that $\abs{I_{\sigma,\delta}}\leq T_2-T_1$ and the estimate on $N_f(R)$ from \cref{eq:zeros-and-product-estimates}, we can estimate
    $$J_\delta(\sigma)\leq\left((T_2-T_1)\alpha+2C\int_0^1\log\frac{1}{\tau}\,\dd\tau\right)\log(\norm{f}_\Omega+1).$$
    In particular, we see that there exists an $M_0>0$, depending only on $\Omega,R$ and $m$, such that \cref{eq:log-integral-epsilon} holds if $\norm{f}_\Omega>M_0$ for any choice of $\delta$, and so we may from here on suppose that $\norm{f}_\Omega\leq M_0$. Set $N_0=C\log(M_0+1)$ and note that then $N_f(R)\leq N_0$. Define $r_\delta$ by
    $$\delta=\frac{r_\delta^{N_0}}{(M_0+1)^\alpha}.$$
    Note that $r_\delta$ decreases to $0$ as $\delta\downarrow0$. Fix now $\delta_0\in(0,1)$ to be chosen later, and let $\delta\in(0,\delta_0]$. Observe that if $t\in I_{\sigma,\delta}$, then $\dist(\sigma+it,\sZ(f)\cap R)\leq r_\delta$ as a consequence of the second estimate in \cref{eq:zeros-and-product-estimates}, and so
    $$\abs{I_{\sigma,\delta}}\leq N_f(R)r_\delta\leq Cr_{\delta_0}\log(\norm{f}_\Omega+1).$$
    Choosing $\delta_0$ sufficiently small such that $r_{\delta_0}<1$, we can use this together with \cref{eq:j-delta-first-estimate} and the inequality
    $$\int_{I_{\sigma,\delta}}\log\frac{1}{\abs{t-\Im s'}}\,\dd t\leq 2N_f(R)\int_0^{r_{\delta_0}}\log\frac{1}{\tau}\,\dd\tau$$
    to estimate
    $$J_\delta(\sigma)\leq \left(\alpha r_{\delta_0}+2C\int_0^{r_{\delta_0}}\log\frac{1}{\tau}\,\dd\tau\right)C\log^2(\norm{f}_\Omega+1).$$
    From here the result is clear.
\end{proof}

All the above lemmas have been purely of a function theoretic nature, and so one is naturally led to wonder where the properties of our situation enter. As mentioned before, the goal is to control our functions in a way that is invariant under vertical translations. In view of the above lemmas, we thus require the existence of a constant $m>0$ such that
$$\norm{V_\tau f}_R\geq m$$
for all $\tau\in\bR$, for a well-chosen rectangle $R$, where we write
$$V_\tau f(s)=f(s+i\tau)$$
for the vertical translations of $f$. This is where almost periodicity enters. In particular, our functions are almost periodic on $\bC_{1/2+\varepsilon}$ for all $\varepsilon>0$, and so we should choose $R$ sufficiently large so that it intersects such a half-plane. If we do this, then the control we need comes from the following lemma. We will use the notation
$$\bS_{\alpha,\beta}=\{s\in\bC:\alpha<\Re s<\beta\}$$
for vertical strips. Almost periodic functions in vertical strips are defined in the natural way.

\begin{lemma}\label{lemma:almost-periodic-sup-bound-below}
    Let $f$ be an analytic almost periodic function in the vertical strip $\bS_{\alpha,\beta}$ that is not identically zero, and let $R=(\alpha,\beta)\times(T_1,T_2)$ for some $T_1<T_2$. Then there exists an $m>0$ such that
    $$\norm{V_\tau f}_R\geq m$$
    for all $\tau\in\bR$.
\end{lemma}

\begin{proof}
    Suppose for a contradiction that such an $m$ does not exist. Then we can find a sequence $\{\tau_n\}_{n\in\bZ^+}$ of real numbers such that
    \begin{equation}\label{eq:vertical-max-estimate}
    \norm{V_{\tau_n}f}_R<\frac{1}{n}
    \end{equation}
    for all $n\in\bZ^+$. By almost periodicity, we can find a subsequence $\{\tau_{n_k}\}_{k\in\bZ^+}$ such that $\{V_{\tau_{n_k}}f\}_{k\in\bZ^+}$ converges uniformly on $\bS_{\alpha,\beta}$, by \cref{eq:vertical-max-estimate} the limit function must be identically zero in $R$, and by analytic continuation it has to be identically zero in all of $\bS_{\alpha,\beta}$. But then
    $$\norm{f}_{\bS_{\alpha,\beta}}=\lim_{k\to\infty}\norm{V_{\tau_{n_k}}f}_{\bS_{\alpha,\beta}}=0,$$
    contradicting that $f$ is not identically zero.
\end{proof}

We will in particular need the following immediate consequence of \cref{lemma:almost-periodic-sup-bound-below}; the details are left to the reader.

\begin{lemma}\label{lemma:sup-bound-below-sequence}
    Let $f$ be an analytic almost periodic function in the vertical strip $\bS_{\alpha,\beta}$ that is not identically zero, and let $\{f_n\}_{n\in\bZ^+}$ be a sequence of analytic almost periodic functions in $\bS_{\alpha,\beta}$, none identically zero, that converges uniformly to $f$ in $\bS_{\alpha,\beta}$, and let $R=(\alpha,\beta)\times(T_1,T_2)$ for some $T_1<T_2$. Then there exists an $m>0$ such that
    $$\norm{V_\tau f}_R\geq m,\quad\norm{V_\tau f_n}_R\geq m$$
    for all $\tau\in\bR$ and all $n\in\bZ^+$.
\end{lemma}

The last main ingredient is the following lemma on mean values of maxima over compact sets. For notational convenience, if $f$ is a complex-valued function on the real line, then we write
$$M_p(f)=\left(\lim_{T\to\infty}\frac{1}{2T}\int_{-T}^T\abs{f(t)}^p\,\dd t\right)^{1/p}$$
for its $p$-mean whenever it exists. We shall also write $\norm{Vf}_\Omega$ for the mapping $\tau\mapsto\norm{V_\tau f}_\Omega$.

\begin{lemma}\label{lemma:sup-mean-values}
    Let $1\leq p<\infty$, let $f$ be an analytic function in a vertical strip $\bS_{\alpha,\beta}$, let $\{f_n\}_{n\in\bZ^+}$ be a sequence of analytic almost periodic functions in $\bS_{\alpha,\beta}$, and suppose that
    $$\lim_{n\to\infty}\sup_{\alpha<\sigma<\beta}\limsup_{T\to\infty}\frac{1}{2T}\int_{-T}^T\abs{f_n(\sigma+it)-f(\sigma+it)}^p\,\dd t=0.$$
    If $K$ is a compact subset of $\bS_{\alpha,\beta}$, then the $p$-means
    $$M_p(\norm{Vf}_K),\quad M_p(\norm{Vf_n}_K),\quad M_p(\norm{Vf_n-Vf}_K)$$
    exist for all $n\in\bZ^+$, and
    $$\lim_{n\to\infty}M_p(\norm{Vf_n}_K)=M_p(\norm{Vf}_K),\quad\lim_{n\to\infty}M_p(\norm{Vf_n-Vf}_K)=0.$$
\end{lemma}

\begin{proof}
    As $K$ is a compact subset of $\bS_{\alpha,\beta}$, we can find a closed rectangle $R=[a,b]\times[-T,T]$ with $K\subseteq R\subseteq\bS_{\alpha,\beta}$ such that $\delta=\dist(K,\partial R)>0$. Using the sub-mean value inequality, we can then estimate
    \begin{multline*}
        \abs{V_\tau f_n(s)-V_\tau f(s)}^p\leq\frac{1}{\pi\delta^2}\iint\limits_{B_\delta(s)}\abs{V_\tau f_n(\sigma+it)-V_\tau f(\sigma+it)}^p\,\dd\sigma\,\dd t \\
        \leq\frac{1}{\pi\delta^2}\iint_R\abs{V_\tau f_n(\sigma+it)-V_\tau f(\sigma+it)}^p\,\dd\sigma\,\dd t
    \end{multline*}
    for all $s\in K$ and $\tau\in\bR$. Taking a supremum over $s\in K$ and integrating, it follows from this, together with Fubini's theorem and the assumptions, that
    \begin{equation}\label{eq:mean-value-limsup-zero}
    \lim_{n\to\infty}\limsup_{T\to\infty}\frac{1}{2T}\int_{-T}^T\norm{V_\tau f_n-V_\tau f}_K^p\,\dd\tau=0.
    \end{equation}
    Next, as $\babs{\norm{V_\tau f_n}_K-\norm{V_\tau f}_K}\leq\norm{V_\tau f_n-V_\tau f}_K$ for all $\tau\in\bR$, it follows from \cref{eq:mean-value-limsup-zero} and the reverse triangle inequality that
    \begin{equation}\label{eq:mean-value-different-limsup-zero}
    \lim_{n\to\infty}\limsup_{T\to\infty}\Abs{\left(\frac{1}{2T}\int_{-T}^T\norm{V_\tau f_n}_K^p\,\dd\tau\right)^{1/p}-\left(\frac{1}{2T}\int_{-T}^T\norm{V_\tau f}_K^p\,\dd\tau\right)^{1/p}}=0.
    \end{equation}
    As $f_n$ is almost periodic for each $n\in\bZ^+$, the function $\norm{Vf_n}_K$ is also almost periodic (as a function on the real line), and as a consequence $M_p(\norm{Vf_n}_K)$ always exists. Combining this with \cref{eq:mean-value-different-limsup-zero} it follows that $M_p(\norm{Vf}_K)$ exists, and that
    $$\lim_{n\to\infty}M_p(\norm{Vf_n}_K)=M_p(\norm{Vf}_K).$$
    For the last assertion, one can argue similarly after observing that
    $$\lim_{m\to\infty}\limsup_{T\to\infty}\Abs{\left(\frac{1}{2T}\int_{-T}^T\norm{V_\tau f_n-V_\tau f_m}_K^p\,\dd\tau\right)^{1/p}-\left(\frac{1}{2T}\int_{-T}^T\norm{V_\tau f_n-V_\tau f}_K^p\,\dd\tau\right)^{1/p}}=0.$$
\end{proof}

We can now prove the main lemma for the proof of the existence of the Jessen function. This lemma is essentially baked into the proof of \cite{borchsenius_mean_1948}*{Theorem 1}; however, we adapt it to our situation using \cref{lemma:riesz-means-convergence}. It will both play a role in the proof of \cref{theorem:existence-of-jessen-function}, and later in the proof of the Hardy--Stein identity in the case of $p=1$.

\begin{lemma}\label{lemma:delta-log-limits}
    Let $1\leq p<\infty$ and let $f$ be a somewhere convergent Dirichlet series that is not identically zero with analytic continuation to $\bC_0$ satisfying
    $$\sup_{\sigma>0}\sup_{T\geq1}\frac{1}{2T}\int_{-T}^T\abs{f(\sigma+it)}^p\,\dd t<\infty.$$
    Then, for all $0<\alpha<\beta$,
    $$\lim_{\delta\downarrow0}\limsup_{T\to\infty}\sup_{\alpha<\sigma<\beta}\frac{1}{2T}\int_{-T}^T(\log\abs{f(\sigma+it)}_\delta-\log\abs{f(\sigma+it)})\,\dd t=0$$
    and, for $k>3$,
    $$\lim_{\delta\downarrow0}\limsup_{N\to\infty}\limsup_{T\to\infty}\sup_{\alpha<\sigma<\beta}\frac{1}{2T}\int_{-T}^T(\log\abs{R_N^kf(\sigma+it)}_\delta-\log\abs{R_N^kf(\sigma+it)})\,\dd t=0.$$
\end{lemma}

\begin{proof}
    Fix $0<\alpha<\beta$ and $k>3$. Suppose without loss of generality that $\alpha<1<\beta$ and that for all $N\geq2$, the function $R_N^kf$ is not identically zero. By \cref{lemma:sup-bound-below-sequence}, there exists a $m>0$ such that
    $$\norm{V_\tau f}_{[1,\beta]\times[0,1]}\geq m,\quad\norm{V_\tau R_N^kf}_{[1,\beta]\times[0,1]}\geq m$$
    for all $\tau\in\bR$ and all $N\geq2$. Let $K$ be any compact subset of $\bC_0$ containing $R=[\alpha,\beta]\times[0,1]$ and such that $\dist(R,\partial K)>0$, and apply \cref{lemma:log-estimate-c-family} to find a family $\{C_\delta\}_{\delta>0}$ of positive numbers with $C_\delta\to0$ as $\delta\downarrow0$, depending only on $K$, $R$ and $m$, with which we can estimate
    \begin{equation}\label{eq:unit-interval-integral-bound}
        \int_{\tau}^{\tau+1}(\log\abs{f(\sigma+it)}_\delta-\log\abs{f(\sigma+it)})\,\dd t\leq C_\delta\log^2(\norm{V_\tau f}_K+1)
    \end{equation}
    and
    $$\int_{\tau}^{\tau+1}(\log\abs{R_N^kf(\sigma+it)}_\delta-\log\abs{R_N^kf(\sigma+it)})\,\dd t\leq C_\delta\log^2(\norm{V_\tau R_N^kf}_K+1)$$
    for all $\tau\in\bR$, all $\sigma\in[\alpha,\beta]$ and all $N\geq2$. Set $D=\sup_{x\geq m}\log^2(x+1)/x^p<\infty$ and $D_\delta=C_\delta D$ for $\delta>0$, and use Fubini's theorem together with \cref{eq:unit-interval-integral-bound} to estimate
    \begin{align*}
    \frac{1}{2T}\int_{-T}^T(\log\abs{f(\sigma+it)}_\delta-\log&\abs{f(\sigma+it)})\,\dd t \\
        &\leq\frac{1}{2T}\int_{-T-1}^T\int_\tau^{\tau+1}(\log\abs{f(\sigma+it)}_\delta-\log\abs{f(\sigma+it)})\,\dd t\,\dd\tau \\
        &\leq D_\delta\frac{T+1}{T}\frac{1}{2(T+1)}\int_{-(T+1)}^{T+1}\norm{V_\tau f}_K^p\,\dd\tau
    \end{align*}
    for all $\sigma\in[\alpha,\beta]$. By estimating similarly for the Riesz means and applying \cref{lemma:sup-mean-values}, it then follows that
    $$\limsup_{T\to\infty}\sup_{\alpha<\sigma<\beta}\frac{1}{2T}\int_{-T}^T(\log\abs{f(\sigma+it)}_\delta-\log\abs{f(\sigma+it)})\,\dd t\leq D_\delta M_p^p(\norm{Vf}_K)$$
    and
    $$\limsup_{N\to\infty}\limsup_{T\to\infty}\sup_{\alpha<\sigma<\beta}\frac{1}{2T}\int_{-T}^T(\log\abs{R_N^kf(\sigma+it)}_\delta-\log\abs{R_N^kf(\sigma+it)})\,\dd t\leq D_\delta M_p^p(\norm{Vf}_K).$$
    Letting $\delta\downarrow0$, the result follows.
\end{proof}

Before proving \cref{theorem:existence-of-jessen-function}, we record the following useful consequence of \cref{lemma:delta-log-limits} which we shall need later.

\begin{lemma}\label{lemma:small-value-density}
    Let $1\leq p<\infty$ and let $f$ be a somewhere convergent Dirichlet series that is not identically zero with analytic continuation to $\bC_0$ satisfying
    $$\sup_{\sigma>0}\sup_{T\geq1}\frac{1}{2T}\int_{-T}^T\abs{f(\sigma+it)}^p\,\dd t<\infty.$$
    Then, for all $0<\alpha<\beta$,
    $$\lim_{\delta\downarrow0}\limsup_{T\to\infty}\sup_{\alpha<\sigma<\beta}\frac{\abs{\{t\in[-T,T]:\abs{f(\sigma+it)}\leq\delta\}}}{2T}=0$$
    and, for $k>3$,
    $$\lim_{\delta\downarrow0}\limsup_{N\to\infty}\limsup_{T\to\infty}\sup_{\alpha<\sigma<\beta}\frac{\abs{\{t\in[-T,T]:\abs{R_N^kf(\sigma+it)}\leq\delta\}}}{2T}=0.$$
\end{lemma}

\begin{proof}
    This follows immediately from \cref{lemma:delta-log-limits} together with the inequality
    $$\mathbf 1_{\overline{\delta\bD}}(z)\leq\frac{\log\abs{z}_{2\delta}-\log\abs{z}}{\log 2}$$
    for all $z\in\bC$ and $\delta>0$.
\end{proof}

We can now prove \cref{theorem:existence-of-jessen-function}.

\begin{proof}[Proof of \cref{theorem:existence-of-jessen-function}]
    Fix $0<\alpha<\beta$ and $k>3$. Suppose without loss of generality that for all $N\geq2$, the function $R_N^kf$ is not identically zero. For $\delta>0$, use that $z\mapsto\log\abs{z}_\delta$ is $1/\delta$-Lipschitz together with Hölder's inequality to estimate
    \begin{multline*}
    \frac{1}{2T}\int_{-T}^T\abs{\log\abs{R_N^kf(\sigma+it)}_\delta-\log\abs{f(\sigma+it)}_\delta}\,\dd t \\
        \leq\frac{1}{\delta}\left(\frac{1}{2T}\int_{-T}^T\abs{R_N^kf(\sigma+it)-f(\sigma+it)}^p\,\dd t\right)^{1/p}.
    \end{multline*}
    Taking suprema over $\sigma>\alpha$ and $T\geq1$, and using \cref{lemma:riesz-means-convergence}, it then follows that
    \begin{equation}\label{eq:riesz-means-log-delta}
    \lim_{N\to\infty}\sup_{\sigma>\alpha}\sup_{T\geq1}\frac{1}{2T}\int_{-T}^T\abs{\log\abs{R_N^kf(\sigma+it)}_\delta-\log\abs{f(\sigma+it)}_\delta}\,\dd t=0.
    \end{equation}
    Estimating now
    \begin{align*}
    \Biggl\lvert\frac{1}{2T}\int_{-T}^T\log\abs{R_N^kf(\sigma+it)}\,&\dd t-\frac{1}{2T}\int_{-T}^T\log\abs{f(\sigma+it)}\,\dd t\Biggr\rvert \\
        &\leq\frac{1}{2T}\int_{-T}^T(\log\abs{R_N^kf(\sigma+it)}_\delta-\log\abs{R_N^kf(\sigma+it)})\,\dd t \\
        &\qquad+\frac{1}{2T}\int_{-T}^T\abs{\log\abs{R_N^kf(\sigma+it)}_\delta-\log\abs{f(\sigma+it)}_\delta}\,\dd t \\
        &\qquad+\frac{1}{2T}\int_{-T}^T(\log\abs{f(\sigma+it)}_\delta-\log\abs{f(\sigma+it)})\,\dd t
    \end{align*}
    and applying \cref{lemma:delta-log-limits} together with \cref{eq:riesz-means-log-delta}, it follows that
    $$\lim_{N\to\infty}\limsup_{T\to\infty}\sup_{\alpha<\sigma<\beta}\Abs{\frac{1}{2T}\int_{-T}^T\log\abs{R_N^kf(\sigma+it)}\,\dd t-\frac{1}{2T}\int_{-T}^T\log\abs{f(\sigma+it)}\,\dd t}=0.$$
    As the Jessen function $\sJ(R_N^kf,\sigma)$ exists uniformly for $\sigma\in(\alpha,\beta)$ by \cite{jessen_uber_1933}*{Satz A}, it follows that $\sJ(f,\sigma)$ exists uniformly for $\sigma\in(\alpha,\beta)$ and that
    $$\lim_{N\to\infty}\sJ(R_N^kf,\sigma)=\sJ(f,\sigma)$$
    uniformly for $\sigma\in(\alpha,\beta)$. That $\sigma\mapsto\sJ(f,\sigma)$ is decreasing and convex is an immediate consequence of this being true of $\sigma\mapsto\sJ(R_N^kf,\sigma)$ for all $N\geq2$. Finally, if $f(+\infty)\neq0$, then we can use almost periodicity together with \cref{lemma:riesz-means-convergence} to find a $\kappa>0$ such that $\abs{f(s)}\geq\abs{f(+\infty)}/2$ for all $s\in\bC_\kappa$, and an $N_0$ such that $\abs{R_N^kf(s)-f(s)}\leq\abs{f(+\infty)}/4$ for all $s\in\bC_\kappa$ and all $N\geq N_0$. Then, if $\delta=\abs{f(+\infty)}/4$, it holds that $\abs{f(s)}_\delta=\abs{f(s)}$ and $\abs{R_N^kf(s)}_\delta=\abs{R_N^kf(s)}$ for all $s\in\bC_\kappa$ and $N\geq N_0$, and so the claim follows from combining the uniform convergence we have shown on intervals with \cref{eq:riesz-means-log-delta}.
\end{proof}

\section{The Hardy--Stein identity and Littlewood--Paley formula}\label{section:hardy-stein-littlewood-paley}

Now that we have obtained some control on the zeros of Dirichlet series, our next goal is to prove a Hardy--Stein identity for the derivative of $\sigma\mapsto M_p^p(f,\sigma)$ for Dirichlet series $f$ satisfying
$$\sup_{\sigma>0}\sup_{T\geq1}\frac{1}{2T}\int_{-T}^T\abs{f(\sigma+it)}^p\,\dd t<\infty.$$
In terms of vertical limits, we will prove the following.

\begin{theorem}[Hardy--Stein identity for $\sH^p$]\label{theorem:hardy-stein-hp}
    Let $f\in\sH^p$ with $1\leq p<\infty$. Then there exists a set $E\subseteq\bT^\infty$ of full measure such that, for all $\chi\in E$, the map $\sigma\mapsto M_p^p(f_\chi,\sigma)$ is continuously differentiable on $(0,\infty)$, and
    $$\partial_\sigma M_p^p(f_\chi,\sigma)=-\lim_{T\to\infty}\frac{p^2}{2T}\int_{-T}^T\int_\sigma^\infty\lvert f_\chi(\nu+it)\rvert^{p-2}\lvert f_\chi'(\nu+it)\rvert^2\,\dd\nu\,\dd t,$$
    where the limit converges uniformly in $\sigma$ on $(\kappa,\infty)$ for all $\kappa>0$. Furthermore, if $p\neq1$, then we may take $E=\sC_p(f)$.
\end{theorem}

The proof consists of two main parts, which use quite different ideas. In the first part we apply Green's theorem to show that if the derivative of $\sigma\mapsto M_p^p(f,\sigma)$ exists and we can calculate it as
$$\partial_\sigma M_p^p(f,\sigma)=\lim_{T\to\infty}\frac{1}{2T}\int_{-T}^T\partial_\sigma\abs{f(\sigma+it)}^p\,\,\dd t,$$
then we have a Hardy--Stein identity for it, and in the second part we approximate our function by its Riesz means to establish that this is indeed the case. Note that the above is always true when the Dirichlet series of $f$ converges uniformly on a half-plane $\bC_\kappa$ with $\kappa<\sigma$ due to almost periodicity, and so our use of Riesz means is precisely to deal with the lack of almost periodicity.

We start with the first part of our proof, which will rely on the following consequence of Green's theorem. The argument is standard, and so we leave it to the reader.

\begin{lemma}\label{lemma:greens-consequence}
    Let $1\leq p<\infty$, let $K\subseteq\bC$ be a compact set such that $\partial K$ is a piecewise continuously differentiable curve, and let $f$ be an analytic function in a neighborhood of $K$. Then
    $$-\iint_K\Delta\abs{f(\sigma+it)}^p\,\dd\sigma\,\dd t=\int_{\partial K}(\partial_t\abs{f(\sigma+it)}^p\,\dd\sigma-\partial_\sigma\abs{f(\sigma+it)}^p\,\dd t).$$
\end{lemma}

In view of this lemma, we shall at several points make use of the identities
\begin{equation}\label{eq:partial-sigma-identity}
    \partial_\sigma\abs{f(\sigma+it)}^p=p\abs{f(\sigma+it)}^{p-2}\Re(f'(\sigma+it)\overline{f(\sigma+it)}),
\end{equation}
\begin{equation}\label{eq:partial-t-identity}
    \partial_t\abs{f(\sigma+it)}^p=-p\abs{f(\sigma+it)}^{p-2}\Im(f'(\sigma+it)\overline{f(\sigma+it)})
\end{equation}
for analytic $f$, $1\leq p<\infty$, and $\sigma+it\in\bC$ with $f(\sigma+it)\neq0$, and so we record them here. They can be shown by direct computation using the Cauchy--Riemann equations.

In our application of Green's theorem we will need to estimate the partial derivative $\partial_t\abs{f(\sigma+it)}^p$ for large values of $\abs{t}$, and so we prove the following lemma.

\begin{lemma}\label{lemma:partial-t-integral-limit}
    Let $1\leq p<\infty$ and let $f$ be a somewhere convergent Dirichlet series with analytic continuation to $\bC_0$ satisfying
    $$\sup_{\sigma>0}\sup_{T\geq1}\frac{1}{2T}\int_{-T}^T\abs{f(\sigma+it)}^p\,\dd t<\infty.$$
    If $\kappa>0$, then $\sigma\mapsto\partial_t\abs{f(\sigma+it)}^p$ is in $L^1([\kappa,\infty))$ for all $t\in\bR$, and
    $$\lim_{T\to\infty}\frac{1}{2T}\int_\kappa^\infty\abs{\partial_t\abs{f(\sigma\pm iT)}^p}\,\dd\sigma=0.$$
\end{lemma}

\begin{proof}
    We consider only the case when the $\pm$ is a $+$. By assumption we can write
    $$f(s)=\sum_{n=1}^\infty a_nn^{-s},\quad f'(s)=-\sum_{n=2}^\infty a_n\log(n)n^{-s}$$
    in some half-plane, and so in particular we can find some $\kappa_0>0$ such that both of the above Dirichlet series converge absolutely in $\overline\bC_{\kappa_0}$ (see, e.g., \cite{titchmarsh_theory_1958}*{Section 9.13}). It suffices to show the result when $\kappa\in(0,\kappa_0)$, so fix such a $\kappa$. For $\sigma+it\in\bC_{\kappa_0}$, estimate
    $$\abs{f'(\sigma+it)}\leq\sum_{n=2}^\infty\abs{a_n}\log (n)n^{-\kappa_0}{n^{-(\sigma-\kappa_0)}}\leq\frac{1}{2^{\sigma-\kappa_0}}\sum_{n=2}^\infty\abs{a_n}\log(n)n^{-\kappa_0}=:\frac{C}{2^\sigma}.$$
    Using this together with the identity \cref{eq:partial-t-identity} and the boundedness of $f$ in $\bC_{\kappa_0}$, it follows that
    \begin{equation}\label{eq:derivative-integral-large-sigma}
        \frac{1}{2T}\int_{\kappa_0}^\infty\abs{\partial_t\abs{f(\sigma+iT)}^p}\,\dd\sigma\leq \frac{1}{2T}Cp\norm{f}_{\bC_{\kappa_0}}^{p-1}\int_{\kappa_0}^\infty\frac{\,\dd\sigma}{2^{\sigma}}\to0
    \end{equation}
    as $T\to\infty$. Use next \cref{lemma:order-strict-inequality,lemma:carlson-derivative-inequality} to find an $r>p$ and a $D>0$ such that
    $$\max\{\abs{f(\sigma+it)},\abs{f'(\sigma+it)}\}\leq D(1+\abs{t})^{1/r}$$
    for all $\sigma+it\in\bC_\kappa$. Using this and the identity \cref{eq:partial-t-identity} again, we have that
    \begin{equation}\label{eq:derivative-integral-small-sigma}
        \frac{1}{2T}\int_{\kappa}^{\kappa_0}\abs{\partial_t\abs{f(\sigma+iT)}^p}\,\dd\sigma\leq pD^p(\kappa_0-\kappa)\frac{(1+T)^{p/r}}{2T}\to0
    \end{equation}
    as $T\to\infty$. The result follows by combining \cref{eq:derivative-integral-large-sigma,eq:derivative-integral-small-sigma}.
\end{proof}

With this estimate, we can now apply Green's theorem to tackle the first part of the proof of the Hardy--Stein identity.

\begin{lemma}\label{lemma:derivative-mean-laplacian-limit}
    Let $1\leq p<\infty$ and let $f$ be a somewhere convergent Dirichlet series with analytic continuation to $\bC_0$ satisfying
    $$\sup_{\sigma>0}\sup_{T\geq1}\frac{1}{2T}\int_{-T}^T\abs{f(\sigma+it)}^p\,\dd t<\infty.$$
    Then
    $$\lim_{T\to\infty}\sup_{\sigma>\kappa}\Abs{-\frac{1}{2T}\int_{-T}^T\int_\sigma^\infty\Delta\abs{f(\nu+it)}^p\,\dd\nu\,\dd t-\frac{1}{2T}\int_{-T}^T\partial_\sigma\abs{f(\sigma+it)}^p\,\dd t}=0$$
    for all $\kappa>0$.
\end{lemma}

\begin{proof}
    Fix $\kappa>0$. Take $\kappa<\sigma<\sigma'$ and $T>0$. Applying \cref{lemma:greens-consequence} to the rectangle $[\sigma,\sigma']\times[-T,T]$ we have that
    \begin{multline}\label{eq:greens-application}
        -\int_{-T}^T\int_\sigma^{\sigma'}\Delta\abs{f(\nu+it)}^p\,\dd\nu\,\dd t=\int_\sigma^{\sigma'}\partial_t\abs{f(\nu-iT)}^p\,\dd\nu-\int_\sigma^{\sigma'}\partial_t\abs{f(\nu+iT)}^p\,\dd\nu \\
        +\int_{-T}^T\partial_\sigma\abs{f(\sigma+it)}^p\,\dd t-\int_{-T}^T\partial_\sigma\abs{f(\sigma'+it)}^p\,\dd t.
    \end{multline}
    As $\abs{f}^p$ is subharmonic, the monotone convergence theorem implies that
    $$\lim_{\sigma'\to\infty}\int_{-T}^T\int_\sigma^{\sigma'}\Delta\abs{f(\nu+it)}^p\,\dd\nu\,\dd t=\int_{-T}^T\int_\sigma^\infty\Delta\abs{f(\nu+it)}^p\,\dd\nu\,\dd t.$$
    Next, estimate
    $$\Abs{\int_{-T}^T\partial_\sigma\abs{f(\sigma'+it)}^p\,\dd t}\leq p2T\norm{f}_{\bC_{\sigma'}}^{p-1}\norm{f'}_{\bC_{\sigma'}}\to0$$
    as $\sigma'\to\infty$ since $f(s)\to f(+\infty)$ and $f'(s)\to f'(+\infty)=0$ uniformly for $s\in\bC_{\sigma'}$ as $\sigma'\to\infty$. The result now follows by letting $\sigma'\to\infty$, dividing by $2T$, letting $T\to\infty$, and using \cref{lemma:partial-t-integral-limit} in \cref{eq:greens-application}.
\end{proof}

The above lemma is sufficiently strong for us to prove a Littlewood--Paley formula.

\begin{theorem}\label{theorem:littlewood-paley-dirichlet-series}
    Let $1\leq p<\infty$ and let $f$ be a somewhere convergent Dirichlet series with analytic continuation to $\bC_0$ satisfying
    $$\sup_{\sigma>0}\sup_{T\geq1}\frac{1}{2T}\int_{-T}^T\abs{f(\sigma+it)}^p\,\dd t<\infty.$$
    Then
    $$\norm{f}_{\sH^p}^p=\abs{f(+\infty)}^p+\lim_{\sigma_0\downarrow0}\lim_{T\to\infty}\frac{p^2}{2T}\int_{-T}^T\int_{\sigma_0}^\infty\abs{f(\sigma+it)}^{p-2}\abs{f'(\sigma+it)}^2(\sigma-\sigma_0)\,\dd\sigma\,\dd t.$$
\end{theorem}

\begin{proof}
    Use first the fundamental theorem of calculus together with the dominated convergence theorem to compute
    $$\abs{f(\sigma_0+it)}^p=\abs{f(+\infty)}^p-\int_{\sigma_0}^\infty\partial_\sigma\abs{f(\sigma+it)}^p\,\dd\sigma,$$
    which can be justified by similar estimates to those in \cref{lemma:partial-t-integral-limit}. An application of Fubini's theorem (which can be justified by the same estimate) then yields that
    $$\frac{1}{2T}\int_{-T}^T\abs{f(\sigma_0+it)}^p\,\dd t=\abs{f(+\infty)}^p-\int_{\sigma_0}^\infty\frac{1}{2T}\int_{-T}^T\partial_\sigma\abs{f(\sigma+it)}^p\,\dd t\,\dd\sigma.$$
    Similarly, by an application of Tonelli's theorem and the identity $\Delta\abs{f}^p=p^2\abs{f}^{p-2}\abs{f'}^2$, one computes
    \begin{multline*}
    \frac{1}{2T}\int_{\sigma_0}^\infty\int_{-T}^T\int_\sigma^\infty\Delta\abs{f(\nu+it)}^p\,\dd\nu\,\dd t\,\dd\sigma \\
        =\frac{p^2}{2T}\int_{-T}^T\int_{\sigma_0}^\infty\abs{f(\sigma+it)}^{p-2}\abs{f'(\sigma+it)}^2(\sigma-\sigma_0)\,\dd\sigma\,\dd t.
    \end{multline*}
    Using the above two computations together with \cref{lemma:derivative-mean-laplacian-limit}, it follows that
    $$M_p^p(f,\sigma_0)=\abs{f(+\infty)}^p+\lim_{T\to\infty}\frac{p^2}{2T}\int_{-T}^T\int_{\sigma_0}^\infty\abs{f(\sigma+it)}^{p-2}\abs{f'(\sigma+it)}^2(\sigma-\sigma_0)\,\dd\sigma\,\dd t.$$
    The result follows by letting $\sigma_0\downarrow0$ in view of \cite{brevig_carlsons_2025}*{Theorem 1}.
\end{proof}

As an immediate corollary of this we have the following.

\begin{corollary}\label{corollary:weak-littlewood-paley-hp}
    Let $f\in\sH^p$ with $1\leq p<\infty$ and let $\chi\in\sC_p(f)$. Then
    $$\norm{f}_{\sH^p}^p=\abs{f(+\infty)}^p+\lim_{\sigma_0\downarrow0}\lim_{T\to\infty}\frac{p^2}{2T}\int_{-T}^T\int_{\sigma_0}^\infty\abs{f_\chi(\sigma+it)}^{p-2}\abs{f_\chi'(\sigma+it)}^2(\sigma-\sigma_0)\,\dd\sigma\,\dd t.$$
\end{corollary}

Combining \cref{corollary:weak-littlewood-paley-hp} with \cref{theorem:ergodic-monotone-convergence}, we can now prove  \cref{theorem:littlewood-paley-hp}.

\begin{proof}[Proof of \cref{theorem:littlewood-paley-hp}]
    For $n\in\bZ^+$, define
    $$F_n(\chi)=p^2\int_{1/n}^\infty\abs{f_\chi(\sigma)}^{p-2}\abs{f_\chi'(\sigma)}^2\left(\sigma-\frac{1}{n}\right)\,\dd\sigma$$
    when $\chi\in\sC_p(f)$ and $F_n(\chi)=0$ when $\chi\in\bT^\infty\setminus\sC_p(f)$. Then each function $F_n$ is non-negative, and $\{F_n\}_{n\in\bZ^+}$ is a non-decreasing sequence. Furthermore, the pointwise limit of the sequence is
    $$F(\chi)=\lim_{n\to\infty}F_n(\chi)=p^2\int_0^\infty\abs{f_\chi(\sigma)}^{p-2}\abs{f_\chi'(\sigma)}^2\sigma\,\dd\sigma$$
    for $\chi\in\sC_p(f)$ by the monotone convergence theorem. Combining \cref{corollary:weak-littlewood-paley-hp,theorem:ergodic-monotone-convergence} we can then compute
    \begin{multline*}
    \norm{f}_{\sH^p}^p=\abs{f(+\infty)}^p+\lim_{n\to\infty}\lim_{T\to\infty}\frac{1}{2T}\int_{-T}^TF_n(\chi\fp^{-it})\,\dd t \\=\abs{f(+\infty)}^p+\lim_{T\to\infty}\frac{1}{2T}\int_{-T}^TF(\chi\fp^{-it})\,\dd t
    \end{multline*}
    for almost every $\chi\in\sC_p(f)$, proving the result.
\end{proof}

This takes care of our application of Green's theorem. For the remaining part we will rely heavily on \cref{lemma:riesz-means-convergence}, and in particular in conjunction with \cref{lemma:carlson-derivative-inequality}, where it allows us to estimate both $f$ and $f'$ by their Riesz means in mean. To do this, we will use the following inequalities.

\begin{lemma}\label{lemma:abs-inequalities}
    Let $a,b\in\bC$. If $p\in(1,2]$, then
    $$\abs{\abs{a}^{p-2}a-\abs{b}^{p-2}b}\leq2^{2-p}\abs{a-b}^{p-1}$$
    and if $p\in[2,\infty)$, then
    $$\abs{\abs{a}^{p-2}a-\abs{b}^{p-2}b}\leq(p-1)(\abs{a}+\abs{b})^{p-2}\abs{a-b}.$$
\end{lemma}

\begin{proof}
    Fix $1<p<\infty$. Suppose without loss of generality that $b=1$ and $a\notin\{0,1\}$. By the fundamental theorem of calculus we can estimate
    $$\abs{\abs{a}^{p-2}a-1}\leq(p-1)\abs{a-1}\int_0^1\abs{t(a-1)+1}^{p-2}\,\dd t.$$
    If $p\geq2$, then the claim follows by an application of the triangle inequality, so suppose that $1<p<2$. Then, using the reverse triangle inequality in the above, we can estimate
    \begin{align*}
    \abs{\abs{a}^{p-2}a-1}
        &\leq(p-1)\abs{a-1}\int_0^1\frac{\dd t}{\abs{t\abs{a-1}-1}^{2-p}} \\
        &=\begin{cases}
        1-(1-\abs{a-1})^{p-1}, &\abs{a-1}\leq1,\\
        1+(\abs{a-1}-1)^{p-1}, &\abs{a-1}>1.
    \end{cases}
    \end{align*}
    For the case when $\abs{a-1}\leq1$, the result follows from the subadditivity of the map $x\mapsto x^{p-1}$, and for the case when $\abs{a-1}>1$, the result follows from the concavity of $x\mapsto x^{p-1}$.
\end{proof}

With this, we show an analogous result to the last part of \cref{lemma:riesz-means-convergence} for the partial derivative $\partial_\sigma\abs{f(\sigma+it)}^p$, which will provide us with precisely what we need. The case of $p=1$ will require stronger assumptions, and so we prove this in a separate lemma after dealing with $1<p<\infty$. This stems from the fact that in the formula
$$\partial_\sigma\abs{f(\sigma+it)}^p=p\abs{f(\sigma+it)}^{p-2}\Re(f'(\sigma+it)\overline{f(\sigma+it)}),$$
the power $p-2$ provides us with sufficiently much cancellation to be able to avoid having to deal with the zero behavior of $f$ if $p>1$, but in the case of $p=1$, the zero behavior of $f$ becomes a genuine issue, and the argument breaks down. To deal with this, we will take an intersection of Carlson sets of horizontal translations to get sufficiently good control.

\begin{lemma}\label{lemma:riesz-means-derivative-limit}
    Let $1<p<\infty$ and let $f$ be a somewhere convergent Dirichlet series with analytic continuation to $\bC_0$ satisfying
    $$\sup_{\sigma>0}\sup_{T\geq1}\frac{1}{2T}\int_{-T}^T\abs{f(\sigma+it)}^p\,\dd t<\infty.$$
    Then, for all $k>3$ and all $\kappa>0$,
    $$\lim_{N\to\infty}\sup_{\sigma>\kappa}\sup_{T\geq1}\frac{1}{2T}\int_{-T}^T\abs{\partial_\sigma\abs{R_N^kf(\sigma+it)}^p-\partial_\sigma\abs{f(\sigma+it)}^p}\,\dd t=0.$$
\end{lemma}

\begin{proof}
    Fix $k>3$ and $\kappa>0$, and use the identity \cref{eq:partial-sigma-identity} to estimate
    \begin{multline*}
    \frac{1}{p}\abs{\partial_\sigma\abs{R_N^kf(s)}^p-\partial_\sigma\abs{f(s)}^p}\\
    \leq\abs{R_N^kf(s)}^{p-1}\abs{R_N^kf'(s)-f'(s)}
        +\abs{f'(s)}\babs{\abs{R_N^kf(s)}^{p-2}R_N^kf(s)-\abs{f(s)}^{p-2}f(s)}.
    \end{multline*}
    For the first term, use \cref{lemma:carlson-derivative-inequality,lemma:riesz-means-convergence} together with Hölder's inequality to get that
    \begin{multline*}
    \sup_{\sigma>\kappa}\sup_{T\geq1}\frac{1}{2T}\int_{-T}^T\abs{R_N^kf(\sigma+it)}^{p-1}\abs{R_N^kf'(\sigma+it)-f'(\sigma+it)}\,\dd t \\
        \leq A_\kappa^{\frac{p-1}{p}}\left(\sup_{\sigma>\kappa}\sup_{T\geq1}\frac{1}{2T}\int_{-T}^T\abs{R_N^kf'(\sigma+it)-f'(\sigma+it)}^p\,\dd t\right)^{1/p} \to 0
    \end{multline*}
    as $N\to\infty$, where
    $$A_\kappa=\sup_{N\geq2}\sup_{\sigma>\kappa}\sup_{T\geq1}\frac{1}{2T}\int_{-T}^T\abs{R_N^kf(\sigma+it)}^p\,\dd t.$$
    For the second term, start by using Hölder's inequality to estimate
    \begin{multline*}
    \frac{1}{2T}\int_{-T}^T\abs{f'(\sigma+it)}\babs{\abs{R_N^kf(\sigma+it)}^{p-2}R_N^kf(\sigma+it)-\abs{f(\sigma+it)}^{p-2}f(\sigma+it)}\,\dd t \\
        \leq B_\kappa^{1/p}\left(\frac{1}{2T}\int_{-T}^T\babs{\abs{R_N^kf(\sigma+it)}^{p-2}R_N^kf(\sigma+it)-\abs{f(\sigma+it)}^{p-2}f(\sigma+it)}^{\frac{p}{p-1}}\,\dd t\right)^{\frac{p-1}{p}}
    \end{multline*}
    where
    $$B_\kappa=\sup_{\sigma>\kappa}\sup_{T\geq1}\frac{1}{2T}\int_{-T}^T\abs{f'(\sigma+it)}^p\,\dd t.$$
    If $p\in(1,2]$, use \cref{lemma:abs-inequalities,lemma:riesz-means-convergence} to get that
    \begin{multline*}
    \sup_{\sigma>\kappa}\sup_{T\geq1}\frac{1}{2T}\int_{-T}^T\babs{\abs{R_N^kf(\sigma+it)}^{p-2}R_N^kf(\sigma+it)-\abs{f(\sigma+it)}^{p-2}f(\sigma+it)}^{\frac{p}{p-1}}\,\dd t \\
    \leq 2^{\frac{(2-p)p}{p-1}}\sup_{\sigma>\kappa}\sup_{T\geq1}\frac{1}{2T}\int_{-T}^T\abs{R_N^kf(\sigma+it)-f(\sigma+it)}^p\,\dd t\to0
    \end{multline*}
    as $N\to\infty$, and if $p\in(2,\infty)$, use \cref{lemma:abs-inequalities,lemma:riesz-means-convergence} together with Hölder's inequality to get that
    \begin{multline*}
    \sup_{\sigma>\kappa}\sup_{T\geq1}\frac{1}{2T}\int_{-T}^T\babs{\abs{R_N^kf(\sigma+it)}^{p-2}R_N^kf(\sigma+it)-\abs{f(\sigma+it)}^{p-2}f(\sigma+it)}^{\frac{p}{p-1}}\,\dd t \\
    \leq (p-1)(C_\kappa+D_\kappa)^{p-2}\left(\sup_{\sigma>\kappa}\sup_{T\geq1}\frac{1}{2T}\int_{-T}^T\abs{R_N^kf(\sigma+it)-f(\sigma+it)}^p\,\dd t\right)^{1/p}\to0
    \end{multline*}
    as $N\to\infty$, where
    $$C_\kappa=\sup_{N\geq2}\sup_{\sigma>\kappa}\sup_{T\geq1}\frac{1}{2T}\int_{-T}^T\abs{R_N^kf(\sigma+it)}^p\,\dd t,\quad D_\kappa=\sup_{\sigma>\kappa}\sup_{T\geq1}\frac{1}{2T}\int_{-T}^T\abs{f(\sigma+it)}^p\,\dd t.$$
    The result now follows after combining the above.
\end{proof}

This takes care of the case when $1<p<\infty$. We now deal with the case when $p=1$.

\begin{lemma}\label{lemma:riesz-means-derivative-limit-p-1}
    Let $f$ be a somewhere convergent Dirichlet series with analytic continuation to $\bC_0$ satisfying
    $$\sup_{\sigma>\kappa}\sup_{T\geq1}\frac{1}{2T}\int_{-T}^T\abs{f(\sigma+it)}^2\,\dd t<\infty$$
    for all $\kappa>0$. Then, for all $k>3$ and all $\kappa>0$,
    $$\lim_{N\to\infty}\sup_{\sigma>\kappa}\sup_{T\geq1}\frac{1}{2T}\int_{-T}^T\babs{\partial_\sigma\abs{R_N^kf(\sigma+it)}-\partial_\sigma\abs{f(\sigma+it)}}\,\dd t=0.$$
\end{lemma}

\begin{proof}
    Fix $\kappa>0$ and $k>3$. Use \cref{eq:partial-sigma-identity} to estimate
    $$\babs{\partial_\sigma\abs{R_N^kf(\sigma+it)}-\partial_\sigma\abs{f(\sigma+it)}}\leq \abs{R_N^kf'(\sigma+it)}+\abs{f'(\sigma+it)}$$
    and
    \begin{multline*}
    \babs{\partial_\sigma\abs{R_N^kf(\sigma+it)}-\partial_\sigma\abs{f(\sigma+it)}} \\
    \leq\abs{R_N^kf'(\sigma+it)-f'(\sigma+it)}+\frac{2}{\abs{f(\sigma+it)}}\abs{R_N^kf'(\sigma+it)}\abs{R_N^kf(\sigma+it)-f(\sigma+it)}
    \end{multline*}
    Splitting the integral into one where $\abs{f(\sigma+it)}<\delta$ and another where $\abs{f(\sigma+it)}\geq\delta$ and using these estimates together with the triangle inequality and Cauchy--Schwarz, we have
    \begin{align*}
    \frac{1}{2T}\int_{-T}^T\babs{\partial_\sigma\abs{R_N^kf(\sigma+it)}-&\partial_\sigma\abs{f(\sigma+it)}}\,\dd t \\
        &\leq \left(B_\kappa^{1/2}+D_\kappa^{1/2}\right)\left(\frac{\abs{\{t\in[-T,T]:\abs{f(\sigma+it)}<\delta\}}}{2T}\right)^{1/2} \\
        &\qquad+\frac{1}{2T}\int_{-T}^T\abs{R_N^kf'(\sigma+it)-f'(\sigma+it)}\,\dd t \\
        &\qquad+\frac{2D_\kappa^{1/2}}{\delta}\left(\frac{1}{2T}\int_{-T}^T\abs{R_N^kf(\sigma+it)-f(\sigma+it)}^2\,\dd t\right)^{1/2}
    \end{align*}
    for $\sigma>\kappa$, $T\geq 1$, $N\geq2$ and $\delta>0$
    where
    $$B_\kappa=\sup_{\sigma>\kappa}\sup_{T\geq1}\frac{1}{2T}\int_{-T}^T\abs{f'(\sigma+it)}^2\,\dd t,\quad D_\kappa=\sup_{N\geq2}\sup_{\sigma>\kappa}\sup_{T\geq1}\frac{1}{2T}\int_{-T}^T\abs{R_N^kf'(\sigma+it)}^2\,\dd t.$$
    The result follows by taking suprema over $T\geq1$ and $\sigma>\kappa$, letting $N\to\infty$ and letting $\delta\downarrow0$ by \cref{lemma:riesz-means-convergence,lemma:carlson-derivative-inequality,lemma:small-value-density}.
\end{proof}

With this, we can now prove a version of the Hardy--Stein identity under the above assumptions.

\begin{theorem}\label{theorem:hardy-stein-dirichlet-series}
    Let $1\leq p<\infty$ and let $f$ be a somewhere convergent Dirichlet series with analytic continuation to $\bC_0$. If $p\neq1$, suppose that
    $$\sup_{\sigma>0}\sup_{T\geq1}\frac{1}{2T}\int_{-T}^T\abs{f(\sigma+it)}^p\,\dd t<\infty$$
    and if $p=1$, suppose that
    $$\sup_{\sigma>\kappa}\sup_{T\geq1}\frac{1}{2T}\int_{-T}^T\abs{f(\sigma+it)}^2\,\dd t<\infty$$
    for all $\kappa>0$. Then the map $\sigma\mapsto M_p^p(f,\sigma)$ is continuously differentiable on $(0,\infty)$, and
    $$\partial_\sigma M_p^p(f,\sigma)=-\lim_{T\to\infty}\frac{p^2}{2T}\int_{-T}^T\int_\sigma^\infty\abs{f(\nu+it)}^{p-2}\abs{f'(\nu+it)}^2\,\dd\nu\,\dd t$$
    where the limit converges uniformly in $\sigma$ on $(\kappa,\infty)$ for all $\kappa>0$.
\end{theorem}

\begin{proof}
    Fix $\kappa>0$. By \cref{lemma:derivative-mean-laplacian-limit} it holds that
    $$\lim_{T\to\infty}\sup_{\sigma>\kappa}\Abs{-\frac{1}{2T}\int_{-T}^T\int_\sigma^\infty\Delta\abs{f(\nu+it)}^p\,\dd\nu\,\dd t-\frac{1}{2T}\int_{-T}^T\partial_\sigma\abs{f(\sigma+it)}^p\,\dd t}=0.$$
    Next, by a similar argument to the Moore--Osgood theorem (e.g., \cite{rudin_principles_1976}*{Theorem 7.11}), \cref{lemma:riesz-means-derivative-limit-p-1,lemma:riesz-means-derivative-limit} imply that the limit
    $$\lim_{T\to\infty}\frac{1}{2T}\int_{-T}^T\partial_\sigma\abs{f(\sigma+it)}^p\,\dd t$$
    converges uniformly in $\sigma$ on $(\kappa,\infty)$. From here the result follows readily by \cite{rudin_principles_1976}*{Theorem 7.17} as we may then interchange the limit and the derivative.
\end{proof}

From this theorem, we can deduce \cref{theorem:hardy-stein-hp}.

\begin{proof}[Proof of \cref{theorem:hardy-stein-hp}]
    If $p\neq 1$, then the result is immediate from \cref{theorem:hardy-stein-dirichlet-series}, so suppose $p=1$. Recall that Helson's inequality \cite{helson_hankel_2006} states that if $g(s)=\sum_{n=1}^\infty b_nn^{-s}$ is in $\sH^1$, then
    $$\left(\sum_{n=1}^\infty\frac{\abs{b_n}^2}{d(n)}\right)^{1/2}\leq\norm{g}_{\sH^1},$$
    where $d(n)$ denotes the number of divisors of $n$. Applying this to the horizontal translations $H_\kappa f(s)=f(s+\kappa)$ with $\kappa>0$ and using Cauchy--Schwarz together with the fact that $d(n)=O(n^\varepsilon)$ as $n\to\infty$ for all $\varepsilon>0$, it follows that $H_\kappa f\in\sH^2$. Consequently, the set
    $$E=\bigcap_{n\in\bZ^+}\sC_2(H_{1/n}f)\subseteq\bT^\infty$$
    has full measure, and $f_\chi$ satisfies the assumptions of \cref{theorem:hardy-stein-dirichlet-series} for all $\chi\in E$, from which the result follows.
\end{proof}

That one has to treat the case $p=1$ with more care seems to stem from duality. Indeed in the proof for the case $1<p<\infty$, the argument repeatedly makes use of Hölder's inequality, which in this range works out fine. In the edge case $p=1$ the Hölder conjugate is $\infty$, which should indicate why the methods fail: duality asks for boundedness, which we in general do not have in the range $0<\Re s\leq1/2$. We suspect that this is a failure of our methods, and so we raise the following question, the answer to which we believe is yes.

\begin{question}
    Can one take $E=\sC_p(f)$ in \cref{theorem:hardy-stein-hp} also in the case $p=1$?
\end{question}

\section{Jensen's formula and the mean counting function}\label{section:jensens-formula-and-mean-counting-function}

We now return to the zero behavior of Dirichlet series, with the end goal of proving \cref{theorem:mean-counting-limit-interchange}. The key to proving this is to first establish a version of Jensen's formula for the Jessen function, the existence of which we dealt with in \cref{section:zeros-and-jessen}. In particular, we will prove the following theorem.

\begin{theorem}\label{theorem:jensen-formula}
    Let $1\leq p<\infty$ and let $f$ be a somewhere convergent Dirichlet series with $f(+\infty)\neq0$ and with analytic continuation to $\bC_0$ satisfying
    $$\sup_{\sigma>0}\sup_{T\geq1}\frac{1}{2T}\int_{-T}^T\abs{f(\sigma+it)}^p\,\dd t<\infty.$$
    Then
    $$\lim_{T\to\infty}\frac{\pi}{T}\sum_{\substack{s\in\sZ(f)\\\abs{\Im s}<T\\\Re s>\sigma_0}}(\Re s-\sigma_0)=\sJ(f,\sigma_0)-\log\abs{f(+\infty)}$$
    where the limit converges uniformly in $\sigma_0$ on $(\kappa,\infty)$ for all $\kappa>0$.
\end{theorem}

To prove this, we will argue through Littlewood's argument principle in the same way as \cite{brevig_mean_2021}*{Lemma 6.1}. We start with a lemma.

\begin{lemma}\label{lemma:sequence-of-tau-log-estimate}
    Let $1\leq p<\infty$, let $f$ be a somewhere convergent Dirichlet series that is not identically zero with analytic continuation to $\bC_0$ satisfying
    $$\sup_{\sigma>0}\sup_{T\geq1}\frac{1}{2T}\int_{-T}^T\abs{f(\sigma+it)}^p\,\dd t<\infty,$$
    let $0<\alpha<\beta$ and let $\varepsilon\in(0,\alpha)$. Set $R=[\alpha-\varepsilon,\beta+\varepsilon]\times[-1-\varepsilon,1+\varepsilon]$. Then there exists a sequence $\{\tau_n\}_{n\in\bZ}$ of real numbers with $\abs{\tau_n-n}\leq1$ and $\tau_{-n}=-\tau_n$ for all $n\in\bZ$ and a $C>0$ such that $f(\sigma+i\tau_n)\neq0$ and
    \begin{equation}\label{eq:imaginary-log-derivative-estimate}
    \Abs{\Im\frac{f'(\sigma+i\tau_n)}{f(\sigma+i\tau_n)}}\leq C\log^2(\norm{V_{n}f}_R+1)
    \end{equation}
    for all $\sigma\in[\alpha,\beta]$ and all $n\in\bZ$.
\end{lemma}

\begin{proof}
    Suppose without loss of generality that $f$ is almost periodic on $\bC_\kappa$ for some $\kappa\in(\alpha,\beta)$. Use \cref{lemma:almost-periodic-sup-bound-below} to find an $m>0$ such that $\norm{V_\tau f}_{[\kappa,\beta]\times[-1,1]}\geq m$ for all $\tau\in\bR$. Set $R'=[\alpha-\varepsilon/2,\beta+\varepsilon/2]\times[-1-\varepsilon/2,1+\varepsilon/2]$ and use \cref{lemma:zero-count-bound} to find a $C>0$, depending only on $R$, $R'$ and $m$, such that
    $$N_{V_nf}(R')\leq C\log(\norm{V_nf}_R+1)$$
    for all $n\in\bZ$. For $n\in\bZ$ and $\delta>0$, set
    $$E_{\delta,n}=\bigcup_{s'\in\sZ(V_nf)\cap R'}(\Im s'-\delta,\Im s'+\delta).$$
    Fix a non-negative integer $n$, and let
    $$\delta_{\pm n}=\frac{1}{4C\log(\norm{V_{\pm n}f}_R+1)}$$
    Then
    $$\abs{E_{\delta_n,n}\cup(-E_{\delta_{-n},-n})}\leq2\delta_nN_{V_nf}(R')+2\delta_{-n}N_{V_{-n}f}(R')\leq1$$
    and so we may take $t_n\in[-1,1]\setminus (E_{\delta_n,n}\cup(-E_{\delta_{-n},-n}))$ as this set has positive measure. Set $\tau_n= t_n+n$ and $\tau_{-n}=-\tau_n$. We will show from here that \cref{eq:imaginary-log-derivative-estimate} holds for $\tau_n$; the corresponding estimate for $\tau_{-n}$ follows by an identical argument. By how $t_n$ was chosen, one estimates
    $$\dist([\alpha,\beta]\times\{t_n\},\partial R'\cup(\sZ(V_nf)\cap R'))\geq\frac{1}{B\log(\norm{V_{n}f}_R+1)}$$
    with $B=2\max\left\{C,\frac{1}{\varepsilon\log(m+1)}\right\}$. Set
    $$r=\frac{1}{2}\dist([\alpha,\beta]\times\{t_n\},\partial R'\cup(\sZ(V_nf)\cap R')).$$
    Clearly $V_n f(s)\neq0$ for all $s\in R'$ with $\dist(s,[\alpha,\beta]\times\{t_n\})<r$, and so we can let $\arg V_n f$ denote a branch of the argument of $V_nf$ on the set of all such $s$. By arguing similarly to \cref{lemma:zero-product-estimate} we can find a constant $D>0$, depending only on $R'$, $R$ and $m$, such that
    $$\abs{\arg V_nf(s)-\arg V_nf(\sigma+it_n)}\leq D\log(\norm{V_n f}_R+1)$$
    for all $\sigma\in[\alpha,\beta]$ and all $s\in R'$ with $\abs{\sigma+it_n-s}<r$. Finally, by a simple consequence of the Borel--Carathéodory inequality and the Cauchy estimates (see, e.g., \cite{titchmarsh_theory_1958}*{Chapter V}), along with the observation that
    $$\Im\frac{f'(\sigma+i\tau_n)}{f(\sigma+i\tau_n)}=\partial_\sigma\arg V_nf(\sigma+it_n)$$
    for all $\sigma\in[\alpha,\beta]$, we can estimate
    $$\Abs{\Im\frac{f'(\sigma+i\tau_n)}{f(\sigma+i\tau_n)}}\leq\frac{4}{\pi r}\sup_{\abs{\sigma+it_n-s}<r}\abs{\arg V_nf(s)-\arg V_nf(\sigma+it_n)}\leq\frac{8BD}{\pi}\log^2(\norm{V_nf}_R+1).$$
    The constant $8BD/\pi$ only depends on $R'$, $R$ and $m$, so the result follows.
\end{proof}

We will use the following consequence of this lemma.

\begin{lemma}\label{lemma:imaginary-part-limit-sequence}
    Let $1\leq p<\infty$ and let $f$ be a somewhere convergent Dirichlet series that is not identically zero with analytic continuation to $\bC_0$ satisfying
    $$\sup_{\sigma>0}\sup_{T\geq1}\frac{1}{2T}\int_{-T}^T\abs{f(\sigma+it)}^p\,\dd t<\infty.$$
    Then, for all $0<\alpha<\beta$, there exists a sequence $\{\tau_n\}_{n\in\bZ}$ of real numbers with $\abs{\tau_n-n}\leq1$ and $\tau_{-n}=-\tau_n$ for all $n\in\bZ$ such that $f(\sigma+i\tau_n)\neq0$ for all $\sigma\in[\alpha,\beta]$ and $n\in\bZ$, and
    $$\lim_{n\to\pm\infty}\frac{1}{2\tau_n}\int_\alpha^\beta\Im\frac{f'(\sigma+i\tau_n)}{f(\sigma+i\tau_n)}(\sigma-\alpha)\,\dd\sigma=0.$$
\end{lemma}

\begin{proof}
    Let $\{\tau_n\}_{n\in\bZ}$ and $C>0$ be as in \cref{lemma:sequence-of-tau-log-estimate} with $\varepsilon=\alpha/2$. Use \cref{lemma:order-strict-inequality} to find a $D>0$ such that
    $$\abs{f(\sigma+it)}\leq D(1+\abs{t})^{1/p}$$
    for all $\sigma+it\in\bC_{\alpha/2}$. Then we can estimate
    $$\Abs{\frac{1}{2\tau_n}\int_\alpha^\beta\Im\frac{f'(\sigma+i\tau_n)}{f(\sigma+i\tau_n)}(\sigma-\alpha)\,\dd\sigma}\leq\frac{C(\beta-\alpha)^2}{2}\frac{\log^2(D(1+\abs{\tau_n})^{1/p}+1)}{\abs{\tau_n}}\to0$$
    as $n\to\pm\infty$.
\end{proof}

With this, we can prove \cref{theorem:jensen-formula}.

\begin{proof}[Proof of \cref{theorem:jensen-formula}]
    Fix $\kappa>0$. As $f(+\infty)\neq0$, we can find a $\beta>\kappa$ such that $\abs{f(s)}\geq\abs{f(+\infty)}/2>0$ for all $s\in\bC_\beta$. Let $\{\tau_n\}_{n\in\bZ}$ be as in \cref{lemma:imaginary-part-limit-sequence} with $\alpha=\kappa$ and use Littlewood's argument principle \cite{littlewood_zeros_1924} (see also \cite{titchmarsh_theory_1958}*{Theorem 3.7}) to write
    \begin{align*}
        2\pi\sum_{\substack{s\in\sZ(f)\\\abs{\Im s}<\tau_n\\\Re s>\sigma_0}}(\Re s-\sigma_0)
        &=\int_{-\tau_n}^{\tau_n}\log\abs{f(\sigma_0+it)}\,\dd t-\int_{-\tau_n}^{\tau_n}\log\abs{f(\beta+it)}\,\dd t \\[-8mm]
        &\qquad+(\beta-\sigma_0)\Re\int_{-\tau_n}^{\tau_n}\frac{f'(\beta+it)}{f(\beta+it)}\,\dd t \\
        &\qquad+\Im\int_{\sigma_0}^\beta\left(\frac{f'(\sigma+i\tau_{-n})}{f(\sigma+i\tau_{-n})}-\frac{f'(\sigma+i\tau_n)}{f(\sigma+i\tau_n)}\right)(\sigma-\sigma_0)\,\dd\sigma
    \end{align*}
    for $n\in\bZ^+$ and $\sigma_0\in(\kappa,\beta)$. For the first term on the right-hand side of the above, it follows by \cref{theorem:existence-of-jessen-function} that if we divide it by $2\tau_n$ and let $n\to\infty$, then it converges to $\sJ(f,\sigma_0)$ uniformly for $\sigma_0\in(\kappa,\beta)$. For the second term, as $f$ is bounded away from zero on a neighborhood of $\overline\bC_\beta$, any fixed branch $\log f$ of the logarithm of $f$ on this neighborhood is analytic and almost periodic, by which we have that
    \begin{multline*}
    \frac{1}{2\tau_n}\int_{-\tau_n}^{\tau_n}\log\abs{f(\beta+it)}\,\dd t=\Re\left(\frac{1}{2\tau_n}\int_{-\tau_n}^{\tau_n}\log f(\beta+it)\,\dd t\right) \\\to \Re \log f(+\infty)=\log\abs{f(+\infty)}
    \end{multline*}
    as $n\to\infty$. For the third term, we can argue similarly by almost periodicity to obtain that
    $$\frac{\beta-\sigma_0}{2\tau_n}\Re\int_{-\tau_n}^{\tau_n}\frac{f'(\beta+it)}{f(\beta+it)}\,\dd t\to(\beta-\sigma_0)\Re\frac{f'(+\infty)}{f(+\infty)}=0$$
    as $n\to\infty$ uniformly for $\sigma_0\in(\kappa,\beta)$. Finally, the last term goes to zero uniformly for $\sigma_0\in(\kappa,\beta)$ after dividing by $2\tau_n$ and letting $n\to\infty$ by how $\{\tau_n\}_{n\in\bZ}$ was chosen. This shows that
    $$\lim_{n\to\infty}\frac{\pi}{\tau_n}\sum_{\substack{s\in\sZ(f)\\\abs{\Im s}<\tau_n\\\Re s>\sigma_0}}(\Re s-\sigma_0)=\sJ(f,\sigma_0)-\log\abs{f(+\infty)}$$
    uniformly for $\sigma_0\in(\kappa,\beta)$. This is easily seen to also be uniform for $\sigma_0\geq\beta$ as the sum is empty for such $\sigma_0$, and the Jessen function converges uniformly to $\log\abs{f(+\infty)}$ by a similar argument through almost periodicity as the above. Finally, for the general limit, take $T>2$ and choose $n\in\bZ^+$ such that $\abs{T-n}\leq1$. Then $\tau_{n-2}\leq T$ and $T\leq\tau_{n+2}$, so that we can estimate
    $$\frac{\pi}{n+1}\sum_{\substack{s\in\sZ(f)\\\abs{\Im s}<\tau_{n-2}\\\Re s>\sigma_0}}(\Re s-\sigma_0)\leq\frac{\pi}{T}\sum_{\substack{s\in\sZ(f)\\\abs{\Im s}<T\\\Re s>\sigma_0}}(\Re s-\sigma_0)\leq\frac{\pi}{n-1}\sum_{\substack{s\in\sZ(f)\\\abs{\Im s}<\tau_{n+2}\\\Re s>\sigma_0}}(\Re s-\sigma_0).$$
    The result now follows by what we have already shown by letting $T\to\infty$ and applying the squeeze theorem since $\tau_n/n\to1$ as $n\to\infty$.
\end{proof}

From this we can now deduce the existence of the mean counting function for $f\in\sH^p$.

\begin{theorem}\label{theorem:existence-of-mean-counting-function}
    Let $f\in\sH^p$ with $1\leq p<\infty$, and let $\chi\in\sC_p(f)$. Then the mean counting function
    $$\sM(f,\zeta)=\lim_{\sigma_0\downarrow0}\lim_{T\to\infty}\frac{\pi}{T}\sum_{\substack{s\in\sZ(f_\chi-\zeta)\\\abs{\Im s}<T\\\Re s>\sigma_0}}(\Re s-\sigma_0)$$
    exists and is finite for all $\zeta\in\bC\setminus\{f(+\infty)\}$, and its value is independent of $\chi\in\sC_p(f)$.
\end{theorem}

\begin{proof}
    Fix $\chi\in\sC_p(f)$ and $\zeta\in\bC\setminus\{f(+\infty)\}$. By \cref{theorem:jensen-formula} it suffices to show that
    $$\lim_{\sigma\downarrow0}\sJ(f_\chi-\zeta,\sigma)$$
    exists and is finite, and that its value is independent of $\chi\in\sC_p(f)$. That it exists is clear since the Jessen function is monotone by \cref{theorem:existence-of-jessen-function}, and from this we also have that
    $$-\infty<\sup_{\sigma>0}\sJ(f_\chi-\zeta,\sigma)=\lim_{\sigma\downarrow0}\sJ(f_\chi-\zeta,\sigma)\leq\lim_{\sigma\downarrow0}M_p(f_\chi-\zeta,\sigma)=\norm{f_\chi-\zeta}_{\sH^p}<\infty$$
    where the last equality follows from \cite{brevig_carlsons_2025}*{Theorem 1}, from which finiteness follows. That its value is independent of $\chi\in\sC_p(f)$ is an immediate consequence of \cref{theorem:existence-of-jessen-function} and that
    $$\sJ(R_N^kf_\chi-\zeta,\sigma)$$
    is independent of $\chi\in\sC_p(f)$ for all $k>3$, $N\geq2$ and $\sigma>0$, which follows from \cite{jessen_uber_1933}*{Zusatz zu Satz I}.
\end{proof}

Our final goal is now to deduce \cref{theorem:mean-counting-limit-interchange}, which we will do by a similar idea to that used by Kouroupis and Perfekt in \cite{kouroupis_composition_2023}*{Section 4.2}. The theorem will follow easily from the above theorem together with \cref{theorem:ergodic-monotone-convergence} and the following lemma.

\begin{lemma}\label{lemma:mean-counting-as-mean-value}
    Let $f$ be an analytic function on $\bC_0$ and let $\sigma_0\geq0$. If one of the limits
    $$\lim_{T\to\infty}\frac{\pi}{2T}\int_{-T}^T\sum_{\substack{s\in\sZ(f)\\-1+t<\Im s<1+t\\\Re s>\sigma_0}}(\Re s-\sigma_0)\,\dd t\qquad\text{and}\qquad\lim_{T\to\infty}\frac{\pi}{T}\sum_{\substack{s\in\sZ(f)\\\abs{\Im s}<T\\\Re s>\sigma_0}}(\Re s-\sigma_0)$$
    exists, then so does the other, and their values are equal.
\end{lemma}

\begin{proof}
    By interchanging sum and integral and using that all terms are non-negative, we can estimate
    \begin{align*}
    \frac{\pi}{2T}\int_{-T}^T\sum_{\substack{s\in\sZ(f)\\-1+t<\Im s<1+t\\\Re s>\sigma_0}}(\Re s-\sigma_0)\,\dd t
        &\leq\frac{\pi}{T}\sum_{\substack{s\in\sZ(f)\\\abs{\Im s}<1+T\\\Re s>\sigma_0}}\frac{\Re s-\sigma_0}{2}\int_{\Im s-1}^{\Im s+1}\,\dd t \\
        &\leq\frac{\pi}{2T}\int_{-T-2}^{T+2}\sum_{\substack{s\in\sZ(f)\\-1+t<\Im s<1+t\\\Re s>\sigma_0}}(\Re s-\sigma_0)\,\dd t
    \end{align*}
    from which it follows that if the first limit exists, then so does the second, and their values are equal. A similar argument also shows that if the second limit exists, then so does the first.
\end{proof}

\begin{proof}[Proof of \cref{theorem:mean-counting-limit-interchange}]
    Fix $\zeta\in\bC\setminus\{f(+\infty)\}$. For $n\in\bZ^+$, set
    $$F_n(\chi)=\pi\sum_{\substack{s\in\sZ(f_\chi-\zeta)\\\abs{\Im s}<1\\\Re s>1/n}}\left(\Re s-\frac{1}{n}\right)$$
    for $\chi\in\sC_p(f)$ and $F_n(\chi)=0$ for $\chi\in\bT^\infty\setminus\sC_p(f)$. Then $\{F_n\}_{n\in\bZ^+}$ is non-decreasing and each $F_n$ is non-negative, so by \cref{theorem:ergodic-monotone-convergence} we can find a set $E\subseteq\bT^\infty$ of full measure such that
    $$\lim_{n\to\infty}\lim_{T\to\infty}\frac{1}{2T}\int_{-T}^TF_n(\chi\fp^{-it})\,\dd t=\lim_{T\to\infty}\frac{1}{2T}\int_{-T}^TF(\chi\fp^{-it})\,\dd t$$
    for all $\chi\in E$, where
    $$F(\chi)=\lim_{n\to\infty}F_n(\chi)=\pi\sum_{\substack{s\in\sZ(f_\chi-\zeta)\\\abs{\Im s}<1\\\Re s>0}}\Re s$$
    for $\chi\in\sC_p(f)$. From this the result is immediate from \cref{theorem:existence-of-mean-counting-function,lemma:mean-counting-as-mean-value} and taking $E\cap\sC_p(f)$ as the desired set of full measure.
\end{proof}

It should be noted that the set of full measure in \cref{theorem:mean-counting-limit-interchange} depends on $\zeta$, and so it would be of interest to know whether one can choose one such subset of $\bT^\infty$ of full measure that works for all $\zeta\in\bC\setminus\{f(+\infty)\}$. We believe this to be true, but are unable to show it.

\begin{question}
    Can one choose the set of full measure in \cref{theorem:mean-counting-limit-interchange} independently of $\zeta\in\bC\setminus\{f(+\infty)\}$?
\end{question}

Having shown the existence of the mean counting function in a generalized sense for $f\in\sH^p$, one is led to ask whether this can be extended to an analogous class to the Nevanlinna class for Dirichlet series. We are not able to answer this question; however, we find it of interest to describe the problem. In \cite{brevig_mean_2021}, Brevig and Perfekt introduced the class $\sN_u$ consisting of all Dirichlet series $f$ that converge uniformly on $\bC_\kappa$ for all $\kappa>0$ such that
$$\limsup_{\sigma\downarrow0}\lim_{T\to\infty}\frac{1}{2T}\int_{-T}^T\log^+\abs{f(\sigma+it)}\,\dd t<\infty,$$
where $\log^+ x=\max\{\log x,0\}$. They showed that for $f\in\sN_u$, the mean counting function exists; this can clearly also be deduced from the same argument as in \cref{theorem:existence-of-mean-counting-function} since $f$ is bounded on $\bC_\kappa$ for any $\kappa>0$. It should be noted that their definition of the mean counting function differs slightly from ours; however, the two are equivalent for almost periodic functions, and the definition we use is more natural in view of the Littlewood--Paley formula and Jensen's formula; this should be compared with \cite{brevig_almost_2025}. The assumption of uniform convergence is rather unsatisfactory for a Nevanlinna class of Dirichlet series, and so in \cite{guo_nevanlinna_2025}, Guo, Ni and Zhou defined the Nevanlinna class of Dirichlet series $\sN$ as the completion of $\sN_u$ in the metric
$$d(f,g)=\norm{f-g}_0$$
where
$$\norm{f}_0=\limsup_{\sigma\downarrow0}\lim_{T\to\infty}\frac{1}{2T}\int_{-T}^T\log(\abs{f(\sigma+it)}+1)\,\dd t$$
for $f,g\in\sN_u$.

\begin{question}
    Let $f\in\sN$. Does there exist a set $E\subseteq\bT^\infty$ of full measure such that if $\chi\in E$, then $f_\chi$ has an analytic continuation to $\bC_0$ for which
    $$\lim_{\sigma_0\downarrow0}\lim_{T\to\infty}\frac{\pi}{T}\sum_{\substack{s\in\sZ(f_\chi-\zeta)\\\abs{\Im s}<T\\\Re s>\sigma_0}}(\Re s-\sigma_0)$$
    exists and is finite for all $\zeta\in\bC\setminus\{f(+\infty)\}$?
\end{question}

\bibliography{references}

@book{queffelec_diophantine_2020,
    edition = {Second},
    series = {Texts and {Readings} in {Mathematics}},
    title = {Diophantine approximation and {Dirichlet} series},
    volume = {80},
    isbn = {978-981-15-9350-5 978-981-15-9351-2 978-93-86279-82-8},
    doi = {10.1007/978-981-15-9351-2},
    publisher = {Hindustan Book Agency, New Delhi; Springer, Singapore},
    author = {Queffelec, Hervé and Queffelec, Martine},
    year = {2020},
    mrnumber = {4241378},
}

@misc{brevig_carlsons_2025,
    title = {Carlson's theorem and vertical limit functions},
    url = {https://arxiv.org/abs/2510.05793},
    doi = {10.48550/arXiv.2510.05793},
    publisher = {arXiv},
    author = {Brevig, Ole Fredrik and Kouroupis, Athanasios},
    year = {2025},
    note = {arXiv:2510.05793},
}

@article{saksman_integral_2009,
    title = {Integral means and boundary limits of {Dirichlet} series},
    volume = {41},
    issn = {0024-6093,1469-2120},
    doi = {10.1112/blms/bdp004},
    number = {3},
    journal = {Bulletin of the London Mathematical Society},
    author = {Saksman, Eero and Seip, Kristian},
    year = {2009},
    mrnumber = {2506825},
    pages = {411--422},
}

@article{brevig_almost_2025,
    title = {Almost periodicity and boundary values of {Dirichlet} series},
    volume = {378},
    issn = {0002-9947,1088-6850},
    doi = {10.1090/tran/9420},
    number = {8},
    journal = {Transactions of the American Mathematical Society},
    author = {Brevig, Ole Fredrik and Kouroupis, Athanasios and Perfekt, Karl-Mikael},
    year = {2025},
    mrnumber = {4929857},
    pages = {5677--5700},
}

@article{bayart_approximation_2016,
    title = {Approximation numbers of composition operators on ${H}^p$ spaces of {Dirichlet} series},
    volume = {66},
    issn = {0373-0956,1777-5310},
    doi = {10.5802/aif.3019},
    number = {2},
    journal = {Université de Grenoble. Annales de l'Institut Fourier},
    author = {Bayart, Frédéric and Queffélec, Hervé and Seip, Kristian},
    year = {2016},
    mrnumber = {3477884},
    pages = {551--588},
}

@article{brevig_mean_2021,
    title = {A mean counting function for {Dirichlet} series and compact composition operators},
    volume = {385},
    issn = {0001-8708,1090-2082},
    doi = {10.1016/j.aim.2021.107775},
    journal = {Advances in Mathematics},
    author = {Brevig, Ole Fredrik and Perfekt, Karl-Mikael},
    year = {2021},
    mrnumber = {4252762},
    pages = {Paper No. 107775, 48},
}

@article{jessen_uber_1933,
    title = {Über die {Nullstellen} einer analytischen fastperiodischen {Funktion}. {Eine} {Verallgemeinerung} der {Jensenschen} {Formel}},
    volume = {108},
    issn = {0025-5831,1432-1807},
    doi = {10.1007/BF01452849},
    number = {1},
    journal = {Mathematische Annalen},
    author = {Jessen, Børge},
    year = {1933},
    mrnumber = {1512861},
    pages = {485--516},
}

@article{littlewood_zeros_1924,
    title = {On the zeros of the {Riemann} zeta-function},
    volume = {22},
    copyright = {https://www.cambridge.org/core/terms},
    issn = {0305-0041, 1469-8064},
    doi = {10.1017/S0305004100014225},
    language = {en},
    number = {3},
    journal = {Mathematical Proceedings of the Cambridge Philosophical Society},
    author = {Littlewood, J. E.},
    month = sep,
    year = {1924},
    pages = {295--318},
}

@book{titchmarsh_theory_1958,
    title = {The theory of functions},
    publisher = {Oxford University Press, Oxford},
    author = {Titchmarsh, E. C.},
    year = {1958},
    mrnumber = {3155290},
}

@article{borchsenius_mean_1948,
    title = {Mean motions and values of the {Riemann} zeta function},
    volume = {80},
    issn = {0001-5962,1871-2509},
    doi = {10.1007/BF02393647},
    journal = {Acta Mathematica},
    author = {Borchsenius, Vibeke and Jessen, Børge},
    year = {1948},
    mrnumber = {27796},
    pages = {97--166},
}

@article{bohr_uber_1913_2,
    title = {Über die {Bedeutung} der {Potenzreihen} unendlich vieler {Variablen} in der {Theorie} der {Dirichletschen} {Reihe} $\sum\frac{a_n}{n^s}$},
    volume = {1913},
    url = {https://eudml.org/doc/58885},
    journal = {Nachrichten von der Gesellschaft der Wissenschaften zu Göttingen, Mathematisch-Physikalische Klasse},
    author = {Bohr, H.},
    year = {1913},
    pages = {441--488},
}

@book{besicovitch_almost_1955,
    title = {Almost periodic functions},
    publisher = {Dover Publications, Inc., New York},
    author = {Besicovitch, A. S.},
    year = {1955},
    mrnumber = {68029},
}

@article{bayart_hardy_2002,
    title = {Hardy spaces of {Dirichlet} series and their composition operators},
    volume = {136},
    issn = {0026-9255,1436-5081},
    doi = {10.1007/s00605-002-0470-7},
    number = {3},
    journal = {Monatshefte für Mathematik},
    author = {Bayart, Frédéric},
    year = {2002},
    mrnumber = {1919645},
    pages = {203--236},
}

@article{helson_compact_1969,
    title = {Compact groups and {Dirichlet} series},
    volume = {8},
    issn = {0004-2080,1871-2487},
    doi = {10.1007/BF02589554},
    journal = {Arkiv för Matematik},
    author = {Helson, Henry},
    year = {1969},
    mrnumber = {285858},
    pages = {139--143},
}

@article{aron_dirichlet_2017,
    title = {Dirichlet approximation and universal {Dirichlet} series},
    volume = {145},
    issn = {0002-9939,1088-6826},
    doi = {10.1090/proc/13607},
    number = {10},
    journal = {Proceedings of the American Mathematical Society},
    author = {Aron, Richard M. and Bayart, Frédéric and Gauthier, Paul M. and Maestre, Manuel and Nestoridis, Vassili},
    year = {2017},
    mrnumber = {3690628},
    pages = {4449--4464},
}

@article{kouroupis_composition_2023,
    title = {Composition operators on weighted {Hilbert} spaces of {Dirichlet} series},
    volume = {108},
    issn = {0024-6107,1469-7750},
    doi = {10.1112/jlms.12771},
    number = {2},
    journal = {Journal of the London Mathematical Society. Second Series},
    author = {Kouroupis, Athanasios and Perfekt, Karl-Mikael},
    year = {2023},
    mrnumber = {4626725},
    pages = {837--868},
}

@book{hardy_general_1964,
    series = {Cambridge {Tracts} in {Mathematics} and {Mathematical} {Physics}},
    title = {The general theory of {Dirichlet}'s series},
    volume = {No. 18},
    publisher = {Stechert-Hafner, Inc., New York},
    author = {Hardy, G. H. and Riesz, M.},
    year = {1964},
    mrnumber = {185094},
}

@article{jessen_mean_1945,
    title = {Mean motions and zeros of almost periodic functions},
    volume = {77},
    issn = {0001-5962,1871-2509},
    doi = {10.1007/BF02392225},
    journal = {Acta Mathematica},
    author = {Jessen, Børge and Tornehave, Hans},
    year = {1945},
    mrnumber = {15558},
    pages = {137--279},
}

@book{conway_functions_1978,
    address = {New York, NY},
    series = {Graduate {Texts} in {Mathematics}},
    title = {Functions of {One} {Complex} {Variable} {I}},
    volume = {11},
    copyright = {http://www.springer.com/tdm},
    isbn = {978-0-387-94234-6 978-1-4612-6313-5},
    url = {http://link.springer.com/10.1007/978-1-4612-6313-5},
    doi = {10.1007/978-1-4612-6313-5},
    language = {en},
    urldate = {2025-10-27},
    publisher = {Springer New York},
    author = {Conway, John B.},
    year = {1978},
}

@book{rudin_principles_1976,
    edition = {Third},
    series = {International {Series} in {Pure} and {Applied} {Mathematics}},
    title = {Principles of mathematical analysis},
    publisher = {McGraw-Hill Book Co., New York-Auckland-Düsseldorf},
    author = {Rudin, Walter},
    year = {1976},
    mrnumber = {385023},
}

@article{guo_nevanlinna_2025,
    title = {Nevanlinna class, {Dirichlet} series and {Szegő}'s problem},
    volume = {75},
    issn = {0373-0956,1777-5310},
    doi = {10.5802/aif.3696},
    number = {4},
    journal = {Université de Grenoble. Annales de l'Institut Fourier},
    author = {Guo, Kunyu and Ni, Jiaqi and Zhou, Qi},
    year = {2025},
    mrnumber = {4940712},
    pages = {1603--1647},
}

@article{hedenmalm_hilbert_1997,
    title = {A {Hilbert} space of {Dirichlet} series and systems of dilated functions in ${L}^2(0,1)$},
    volume = {86},
    issn = {0012-7094,1547-7398},
    doi = {10.1215/S0012-7094-97-08601-4},
    number = {1},
    journal = {Duke Mathematical Journal},
    author = {Hedenmalm, Håkan and Lindqvist, Peter and Seip, Kristian},
    year = {1997},
    mrnumber = {1427844},
    pages = {1--37},
}

@article{brevig_volterra_2019,
    title = {Volterra operators on {Hardy} spaces of {Dirichlet} series},
    volume = {754},
    issn = {0075-4102,1435-5345},
    doi = {10.1515/crelle-2016-0069},
    journal = {Journal für die Reine und Angewandte Mathematik. [Crelle's Journal]},
    author = {Brevig, Ole Fredrik and Perfekt, Karl-Mikael and Seip, Kristian},
    year = {2019},
    mrnumber = {4000573},
    pages = {179--223},
}

@article{bayart_counting_2024,
    title = {Counting functions for {Dirichlet} series and compactness of composition operators},
    volume = {56},
    issn = {0024-6093,1469-2120},
    doi = {10.1112/blms.12923},
    number = {1},
    journal = {Bulletin of the London Mathematical Society},
    author = {Bayart, Frédéric},
    year = {2024},
    mrnumber = {4705317},
    pages = {188--197},
}

@article{helson_hankel_2006,
    title = {Hankel forms and sums of random variables},
    volume = {176},
    issn = {0039-3223,1730-6337},
    doi = {10.4064/sm176-1-6},
    number = {1},
    journal = {Studia Mathematica},
    author = {Helson, Henry},
    year = {2006},
    mrnumber = {2263964},
    pages = {85--92},
}

\end{document}